\documentclass{amsart}

\usepackage{color}
\usepackage{amsmath}
\usepackage{amscd,amsfonts,amssymb}
\usepackage[all]{xy}
\usepackage{mathtools}
\usepackage[colorlinks=true]{hyperref}

\def\R{\mathbb{R}}
\def\Q{\mathbb{Q}}

\def\Z{\mathbb{Z}}
\def\N{\mathbb{N}}
\def\O{\mathcal{O}}

\def\deg{{\mathrm{deg}}}

\def\ot{\otimes}

\def\shfL{\mathcal{L}}

\def\ovl{\overline}

\def\scrX{\mathcal{X}}

\usepackage[normalem]{ulem}
\definecolor{mred}{rgb}{0.83, 0.0, 0.0}
\definecolor{darkspringgreen}{rgb}{0.09, 0.45, 0.27}
\definecolor{ruby}{rgb}{0.88, 0.07, 0.37}
\def\colorsout#1{\bgroup\markoverwith{\textcolor{#1}{\rule[0.5ex]{2pt}{0.7pt}}}\ULon} %[0.5ex]{2pt}{0.4pt}
\def\coloruline#1{\bgroup\markoverwith{\textcolor{#1}{\rule[-0.5ex]{2pt}{0.7pt}}}\ULon} %[0.5ex]{2pt}{0.4pt}

\theoremstyle{plain}

\numberwithin{equation}{section}

\makeatletter
\newcommand\tint{\mathop{\mathpalette\tb@int{t}}\!\int}
\newcommand\bint{\mathop{\mathpalette\tb@int{b}}\!\int}
\newcommand\tb@int[2]{%
  \sbox\z@{$\m@th#1\int$}%
  \if#2t%
    \rlap{\hbox to\wd\z@{%
      \hfil
      \vrule width .35em height \dimexpr\ht\z@+1.4pt\relax depth -\dimexpr\ht\z@+1pt\relax
      \kern.05em % a small correction on the top
    }}
  \else
    \rlap{\hbox to\wd\z@{%
      \vrule width .35em height -\dimexpr\dp\z@+1pt\relax depth \dimexpr\dp\z@+1.4pt\relax
      \hfil
    }}
  \fi
}
\makeatother

\usepackage{etoolbox}
\makeatletter
\newcommand*\suppresschapternumber{%
  \let\@makechapterhead\@makeschapterhead
  \patchcmd{\@chapter}
    {\protect\numberline{\thechapter}}
    {}
    {}{}%
}
\newcommand*\removedotbetweenchapterandsection{%
  \renewcommand\thesection{\thechapter\@arabic\c@section}%
}
\makeatother

\newcommand{\mr}[1]{\mathrm{#1}}
\newcommand{\mc}[1]{\mathcal{#1}}

\long\def\comment#1{}

\newtheorem{thm}{Theorem}[section] % number like 3.1, 3.2, 3.3, etc. 
\newtheorem{prop}[thm]{Proposition}
\newtheorem{lem}[thm]{Lemma}
\newtheorem{cor}[thm]{Corollary}
\newtheorem{innercustomthm}{Theorem}
\newenvironment{customthm}[1]
  {\renewcommand\theinnercustomthm{#1}\innercustomthm}
  {\endinnercustomthm}

\newtheorem{thmA}{Theorem}

\theoremstyle{definition} % 'here we change the style' 
\newtheorem{definition}[thm]{Definition} % numbered with thm 

\newtheorem{example}[thm]{Example}

\theoremstyle{remark} % 'style changed again' 

\begin{document}

%----------------------------------------
\title{Arakelov geometry on flag varieties over function fields and related topics}
\author{Yangyu Fan}
\address{Key Laboratory of Algebraic Lie Theory and Analysis of Ministry of Education, School of Mathematics and Statistics, Beijing Institute of Technology, Beijing, 100081, People's Republic of China}
\email{yangyu.fan@bit.edu.cn}

\author{Wenbin Luo}
\address{Beijing International Center for Mathematical Research, Peking University, Beijing 100871, China}
\email{w.luo@bicmr.pku.edu.cn}

\author{Binggang Qu}
\address{Beijing International Center for Mathematical Research, Peking University, Beijing 100871, China}
\email{qubinggang22@bicmr.pku.edu.cn}

%----------------------------------------
\begin{abstract}
Let $k$ be an algebraically closed field of characteristic zero. Let $G$ be a connected reductive group over $k$, $P \subseteq G$ be a parabolic subgroup and $\lambda: P \longrightarrow G$ be a strictly anti-dominant character.  Let $C$ be a projective smooth curve over $k$ with function field $K=k(C)$ and   $F$ be a principal $G$-bundle on $C$. Then  $F/P \longrightarrow C$ is a flag bundle and $\mathcal{L}_\lambda=F \times_P k_\lambda$ on $F/P$ is a relatively ample line bundle.

We compute the height filtration, successive minima, and the Boucksom-Chen concave transform of the height function $h_{\mathcal{L}_\lambda}: X(\overline{K}) \longrightarrow \mathbb{R}$  over the flag variety $X=(F/P)_K$. An interesting application is that the height of $X$ equals to a weighted average of successive minima, and one may view this as a refinement of Zhang's inequality of successive minima.

Let $f \in N^1(F/P)$ be the numerical class of a vertical fiber. We compute the augmented base loci $\mathrm{B}_+(\mathcal{L}_\lambda-tf)$ for any $t \in \mathbb{R}$, and it turns out that they are almost the same as the height filtration. As a corollary, we compute the $k$-th movable cones of flag bundles over curves for all $k$.
\end{abstract}

%----------------------------------------
\maketitle
\setcounter{tocdepth}{1}
\tableofcontents

%----------------------------------------
\section{Introduction}
%----------------------------------------
\subsection{Height filtration and successive minima}
We start by giving a reminder on height functions in Arakelov theory. Let $K$ be either a number field or $K=k(C)$ where $C$ is a projective smooth curve over a field $k$. Let $X$ be a projective variety of dimension $d$ over $K$ and $\overline{L}$ be an adelic line bundle on $X$. These data induce an Arakelov height function $h_{\overline{L}}$ on $X$ (see \cite[\textsection 9]{Yuan_iccm2010} for a survey). A typical case is the \textit{geometric height}, which is the one we concern in this article. Here we give the definition.

If $K=k(C)$, consider a projective flat morphism $\mathcal{X} \longrightarrow C$ with the generic fiber $X \longrightarrow \operatorname{Spec}(K)$ and 
 a line bundle $\mathcal{L}$   on $\mathcal{X}$ with $\mathcal{L}_K \simeq L$. The data $(\mathcal{X},\mathcal{L})$ define an adelic line bundle $\ovl L$ and the height function $h_{\ovl L}$ is given by
        \begin{flalign*}
            \quad \quad \quad  &
                h_{\overline{L}}: X(\overline{K}) \longrightarrow \mathbb{Q}, \; x \longmapsto \frac{\mathcal{L} \cdot \overline{\{x\}}}{\deg(x)} \;\; \text{where $\overline{\{x\}}$ is the closure of $x$ in $\mathcal{X}$}.
            &&
        \end{flalign*}
We also denote this by $h_\shfL$ if there is no ambiguity. If $K$ is a number field, the height function can be defined similarly via arithmetic intersection theory.

For any $t \in \mathbb{R}$, let $Z_t \subseteq X$ be the Zariski closure of the set $\big\{ x \in X(\overline{K}): h_{\overline{L}}(x) < t \big\}$. Let's call $\big\{ Z_t: t \in \mathbb{R} \big\}$ the \emph{height filtration}, and call its jumping points the \emph{successive minima}. Note that our definition of successive minima are slightly different with Zhang \cite{Zhang_smallPoints}. Zhang considers only dimension jumps $e_i=\inf\big\{ t: \dim Z_t \geq d-i+1 \big\}$, $i=1,\dots,d+1$. In this article, $e_i$ will be called the \emph{successive minima of Zhang} to avoid ambiguity.

Determining the height filtration is generally difficult and only two cases are known in the literature:
\begin{itemize}
    \item The Néron-Tate height function on abelian varieties $A$ over $K$. Let $L$ be an ample symmetric line bundle over $A$ endowed with the unique adelic metric such that $[2]^*\overline{L}=2^{\dim A} \overline{L}$.  The height filtration of the induced Néron-Tate height function $\widehat{h}_{L}$ is $A \supsetneq \emptyset$ and the jumping point is $\zeta=0$;
    \item The toric height function on toric varieties \cite{burgos_gil_successive_2015}. They  are given by successively deleting torus orbits. An important example is the projective space $\mathbb{P}^n_\mathbb{Q}$ and the height function induced by the universal bundle $\O_{\mathbb{P}^n_\mathbb{Z}}(1)$ equipped with the Fubini-Study metric $ \|\cdot\|_\text{FS}$. The height filtration on $\mathbb{P}^n_\mathbb{Q}$  is given by \begin{flalign*}
   \quad\quad\quad  \mathbb{P}^n_\mathbb{Q} \supsetneq \bigcup_i \big\{ x_i=0 \big\} \supsetneq \bigcup_{i,j} \big\{ x_i=0,x_j=0 \big\} \supsetneq \cdots \supsetneq \emptyset
    \end{flalign*}
    with successive minima $\frac12 \log(k+1)$, $k=0,1,\dots,n$.
\end{itemize}

 Our first result (Theorem \ref{intro_ thm of ht filtration} below) provides a new case where height filtrations can be explicitly computed. Let $k$ be an algebraically closed field of characteristic zero and $C$ be a projective smooth curve over $k$ with function field $K=k(C)$.
Let $G$ be a connected reductive group over $k$, $P \subseteq G$ be a parabolic subgroup and $\lambda: P \longrightarrow \mathbb{G}_m$ be a strictly antidominant character. Let $F$ be a principal $G$-bundle over $C$ and $\mathcal{X}=F/P$ with generic fiber $X=(F/P)_K$. Let $\mathcal{L}_\lambda = F \times_P k_\lambda$ and $h_{\mathcal{L}_\lambda}: X(\overline{K}) \longrightarrow \mathbb{R}$ be the induced height function.

\[ \xymatrix{
X=(F/P)_K \ar[d] \ar[r] & \mathcal{X}=F/P \ar[d] \\
\operatorname{Spek}(K) \ar[r] & C
} \]

Let $F_Q$ be the canonical reduction of $F$ to $Q$, and $\deg(F_Q) \in X(T)^\vee_\mathbb{Q}$ be the induced cocharacter. Let $W,W_P$ and $W_Q$ be the Weyl group of $G,L_P$ and $L_Q$. For $w \in W_Q \backslash W/W_P$, let $C_w=(F_Q \times_Q QwP/P)_K \subseteq X$ be the corresponding Schubert cell. The numbers $\langle \deg (F_{Q }), w\lambda \rangle$ are well-defined. 
\begin{thmA}[Theorem \ref{thm of ht filtration}] For any $t \in \mathbb{R}$, let $Z_t \subseteq X$ be the Zariski closure of the set $\big\{ x \in X(\overline{K}): h_{\mathcal{L}_\lambda}(x) < t \big\}$. Then
	\begin{flalign*}
    \quad \quad \quad &
    Z_t = X \Bigg\backslash \coprod_{\langle \deg (F_{Q }), w\lambda \rangle\geq t} C_w = \coprod_{\langle \deg (F_{Q }), w\lambda \rangle<t} C_w.
    &&
  \end{flalign*} \label{intro_ thm of ht filtration}
\end{thmA}

In particular, the filtration $\big\{ Z_t:t \in \mathbb{R} \big\}$ is  finite (while in general height filtrations could be infinite) and is given by successively deleting Schubert cells $C_w$ with jumping points $\zeta_w=\langle \deg (F_{Q }), w\lambda \rangle$, $w \in W_Q \backslash W/W_P$.

\subsection{String polytopes and the Boucksom-Chen concave transform}
 Attached to a projective variety $X$ over $K$ and  an adelic line bundle $\overline{L}$ on $X$ with $L$ ample and $\overline{L}$ semipositive, we have the \emph{Okounkov body} $\Delta$ and the Boucksom-Chen concave transform $\varphi : \Delta \longrightarrow \mathbb{R}$ thanks to \cite{Kaveh-Khovanskii,Lazarsfeld_convex_bodies, BouChen11,Ballay21}. The maximum of $\varphi$ equals to the essential minimum of $h_{\overline{L}}$, the minimum of $\varphi$ equals to the absolute minimum of $h_{\overline{L}}$ and the integral mean of $\varphi$ equals to the height of $X$.

In toric cases, they can be expressed as summation of local \emph{roof functions} \cite{Burgos_book_2014}. But in general, the Boucksom-Chen concave transform is difficult to compute. Our second result (Theorem \ref{intro_ concave transform}) provides a new case where the Boucksom-Chen concave transform can be explicitly computed.

To state the result, keep the notations in Theorem \ref{intro_ thm of ht filtration}. Let $w_0 \in W$ be the longest element and fix a reduced decomposition $\underline{w_0}= s_{i_1} s_{i_2} \cdots s_{i_N}$, $N=\dim G/B$. Then we have the string valuation $\nu$ and the string polytope $\Delta_\lambda \subseteq \mathbb{R}^N$ \cite{Littelmann_1998,Bernstein-Zelevinsky}, which is a polytope of dimension $d=\dim G/P$.
Let $\mathcal{P}_\lambda$ be the weight polytope and $p: \Delta_\lambda \longrightarrow \mathcal{P}_\lambda$ be the weight map. The Okounkov body of $(X,L)$ (with respect to a specific valuation) is the string polytope $\Delta_\lambda$.
\begin{thmA}[Theorem \ref{concave transform}] \label{intro_ concave transform}
	The Boucksom-Chen concave transform of ${\mathcal{L}_\lambda}$ is $\Delta_\lambda \longrightarrow \mathbb{R}, x \longmapsto \langle \deg(F_Q),p(x) \rangle$.
\end{thmA}

Let $\mu$ be the Lebesgue measure on the real span of $\Delta_\lambda$, normalized with respect to $\mathbb{Z}^N$. As we mentioned above, the variety height is given by the integral mean of the Boucksom-Chen concave transform \cite{BouChen11}. Thus a direct corollary is
\begin{flalign*}
	\quad \quad \quad &
	h_{\mathcal{L}_\lambda}(X) = \frac{ \mathcal{L}_\lambda^{\dim \mathcal{X}} } {\dim \mathcal{X} \cdot L^{\dim X}} = \frac1{\operatorname{Vol}(\Delta_\lambda)} \int_{\Delta_\lambda} \langle \deg(F_Q),p(x) \rangle \, \mathrm{d}\mu(x).
	&&
\end{flalign*}

The integration can be computed via Weyl symmetry.
\begin{thmA}[Theorem \ref{Zhang type equality}] \label{intro_ Zhang type equality}
    Let $q: W/W_P \longrightarrow W_Q \backslash W/W_P$ be the canoniacl projection. For $w \in W_Q \backslash W/W_P$, let $m_w=\# q^{-1}(w)$. Then 
   \begin{flalign*}
       \quad \quad \quad &
        h_{\mathcal{L}_\lambda} (X) = \frac{1}{\#W/W_P} \sum_{w\in W_Q \backslash W/W_P} m_w \langle \deg(F_Q),w\lambda \rangle .
       &&
   \end{flalign*}
\end{thmA}

Note that $\zeta_w=\langle \deg(F_Q),w\lambda \rangle$ in Theorem \ref{intro_ Zhang type equality} are successive minima of the height function $h_{\mathcal{L}_\lambda}$ on $X$. Thus Theorem \ref{intro_ Zhang type equality} builds an equality between the variety height and a weighted average of the successve mimina. One may recognize this as a refinement of Zhang's inequality of  successive minima \cite{Zhang_smallPoints} 
\begin{flalign*}
    \quad \quad \quad &
    e_1 \geq h_{\mathcal{L}_\lambda}(X) \geq \frac1{d+1} \sum_{i=1}^{d+1} e_i.
    &&
\end{flalign*}
into an equality. Here $e_i$ is the the successive minima of Zhang. 

We remark that it is not possible to build such an equality in general. Consider the height function on $\mathbb{P}^n_\mathbb{Q}$ induced by the  Hermitian line bundle $\big(\mathcal{O}_{\mathbb P^n_\Z}(1),\|\cdot\|_\text{FS} \big)$. The variety height is a rational number $\frac12 \sum_{i=1}^n \sum_{j=1}^i \frac1j$ while the successive minima are irrational numbers $\tfrac12 \log (k+1)$, $k=0,1,\dots,n$.

%----------------------------------------
\subsection{Base loci and movable cones}
It is a general philosophy in Arakelov geometry that the positivity of $\overline{L}$ is equivalent to the positivity of the induced height function $h_{\overline{L}}$. The study of essential minimum and pseudo-effectivity \cite{burgos_gil_successive_2015,Ballay21,qu_arithmetic_2023} provides an evidence. 

Our third result (Theorem \ref{intro_ prop_aug_bas} below) provides a new evidence to support the above mentioned philosophy. Keep the notations in Theorem \ref{intro_ thm of ht filtration}. Let $f \in N^1(F/P)$ be the class of a vertical fiber and $\mathrm{B}_+(\shfL_\lambda-tf)$ be the augmented base locus of $\shfL_\lambda-tf$ for $t \in \mathbb{R}$. They measure the positivity of the line bundle $\mathcal{L}_\lambda$. 

\begin{thmA}[Theorem \ref{prop_aug_bas}] \label{intro_ prop_aug_bas}
\begin{flalign*}
    \quad \quad \quad &
\mr B_+(\shfL_\lambda-tf)=\scrX \Bigg \backslash \coprod_{\langle \deg(F_Q),w\lambda \rangle > t} \mc C_w = \coprod_{\langle \deg(F_Q),w\lambda \rangle \leq t} \mc C_w.
    &&
\end{flalign*}
\end{thmA}

The height filtration $Z_t$ in Theorem \ref{intro_ thm of ht filtration} measures the positivity of the height function $h_{\mathcal{L}_\lambda}$, and we see that $\mr B_+(\shfL_\lambda-tf)$ is just the Zariski closure of $Z_t$ in $\mathcal{X}$ for every $t \not= \zeta_w$. At critical points, $\mr B_+(\shfL_\lambda-tf)$ is upper continuous while $Z_t$ is lower continuous.

A class $\alpha \in N^1(F/P)_\mathbb{R}$ is called $k$-\emph{movable} if $\mr B_+(\alpha)$ has codimension $\geq k$. As a direct corollary of Theorem \ref{intro_ prop_aug_bas}, we compute the movable cones on flag bundles. (Please see \textsection \ref{subsection_ movable cones} for the notations.)
\begin{thmA}[Theorem \ref{mov cone of flag bundle}] \label{intro_ mov cone of flag bundle}
    The $k$-th movable cone $\operatorname{Mov}^k(F/P)$ coincides with the cone  defined by 
    \begin{enumerate}
        \item $\langle \alpha^\vee, \cdot \rangle <0$ for any $\alpha \in \Delta \backslash \Delta_P$, and
        \item $\langle \deg(F_Q),w \cdot \rangle>0$ for any $w \in W/W_P$ with $\ell(w) \geq n-k+1$.
    \end{enumerate}
\end{thmA}

Let $E$ be a vector bundle on a curve $C$. The ample cone of $\mathbb{P}(E)$ is computed in \cite{Miyaoka_chern_classes} and the big cone of $\mathbb{P}(E)$ is computed in \cite{Fulger_positive_cone}. Fulger and Lehmann computed $\operatorname{Mov}^2(\mathbb{P}(E))$ in \cite{FulgerLehmann_Zariski}.
In \cite{Biswas_pseff_cone_grassm}, the big cone of the Grassmann bundle of $E$ is computed and in \cite{Biswas_nef_cone_flag}, the ample cone of flag bundles of $E$ is computed.
Our Theorem \ref{intro_ mov cone of flag bundle} is a simultaneous generalization of all these results.

\begin{thmA}[Corollary \ref{k-movable=ZhangMinima}] \label{intro_ k-movable=ZhangMinima}
    Let $\mathcal{L}$ be a line bundle on $F/P$. Then $\mathcal{L}$ is $k$-movable if and only if $e_{k}(h_{\mathcal{L}})>0$. Here $e_{k}(h_{\mathcal{L}})$ is the $k$-th minimum of Zhang of the height function $h_{\mathcal{L}}$.
\end{thmA}
Take $k=1$, we recover \cite[Theorem 1.1]{Ballay21} (in this special case) that $\mathcal{L}_\lambda$ big $\Longleftrightarrow$ $\zeta_{\operatorname{ess}}(h_{\mathcal{L}_\lambda})>0$.

\subsection{Example: Grassmann bundles} \label{example_grassmann}
We provide here the computation of the essential minimum and the big cone on Grassmann bundles.

Let $k$ be an algebraically closed field and $C$ be a curve over $k$ with function field $K=k(C)$. Let $E$ be a vector bundle of rank $n$ on $C$ with Harder-Narasimhan filtration $0 \subsetneq E_1 \subsetneq \cdots \subsetneq E_r$. Let $B$ be the subgroup of $\operatorname{GL}_n$ of all upper triangular matrices, and let $Q$ be the parabolic subgroup containing $B$ of type $(\operatorname{rank}(E_1),\operatorname{rank}(E_2),\dots,\operatorname{rank}(E_r))$. Here we say a parabolic subgroup containing $B$ is of type $(a_1,\dots,a_r)$ if it consists of upper block trianguluar matrices of the shape
\[
\left[ 
\begin{array}{cccc} 
  A_1 & * &  \cdots & * \\ 
   & A_2 & \cdots & * \\
   &  & \cdots & * \\
   &  &  & A_r 
\end{array} 
\right] , \quad A_i \in \operatorname{GL}_{a_i}.
\]

Let $F(E)$ be the frame bundle of $E$, which is a $\operatorname{GL}_n$-principal bundle. Then the canonical reduction of $F(E)$  to $Q$ is $F_Q$,  the $Q$-bundle of frames respecting  the filtration $0 \subsetneq E_1 \subsetneq \cdots \subsetneq E_r$. Let $\mu_i$ ($1\leq i\leq n$) be the $i$-th number in the sequence
\begin{flalign*}
    \quad \quad \quad & \underbrace{\mu(E_1),\dots,\mu(E_1)}_\text{$\operatorname{rank}(E_1)$ times}, \underbrace{\mu(E_2/E_1),\dots,\mu(E_2/E_1)}_\text{$\operatorname{rank}(E_2/E_1)$ times}, \dots, \underbrace{\mu(E/E_{r-1}),\dots,\mu(E/E_{r-1})}_\text{$\operatorname{rank}(E/E_{r-1})$ times}.
    &&
\end{flalign*}
Note that $X(T)_\mathbb{Q}=\bigoplus_{i=1}^n \mathbb{Q}\lambda_i$ where $\lambda_i: T \longrightarrow \mathbb{G}_m$ is the character of taking the $i$-th entry, and $\langle \deg(F_Q),\lambda_i \rangle=\mu_i$.

Let $\operatorname{Gr}_r(E)=F(E)/P$  be the Grassmann bundle of $r$-dimension quotients, where  $P$ is the parabolic subgroup containing $B$ of type $(n-r,r)$. Identify the Weyl group of $G$ with $S_n$ and  of $P$ with $S_{n-r}\times S_r$ as usual and  $W/W_P=S_n \big/ (S_{n-r}\times S_r)$  with the multi-index set $\big\{ (i_1,\dots,i_r): 1 \leq i_1<i_2<\cdots<i_r \leq n \big\}$ via
\begin{flalign*}
\quad \quad \quad &
\left[ 
\begin{matrix}
  1 & 2 & \cdots & n-r & n-r+1  & \cdots & n \\ 
  * & * & \cdots & * & i_1  & \cdots & i_r
\end{matrix} 
\right] \longleftrightarrow (i_1,\dots,i_r) .
&&
\end{flalign*}

$N^1(\operatorname{Gr}_r(E))$ is two dimensional with basis $\mathcal{O}(1)$ (determinant bundle of the universal bundle on $\operatorname{Gr}_r(E)$) and $f$ (a vertical fiber). Let $\det_2$ be the character 
\begin{flalign*}
    \quad \quad \quad &
    \det\nolimits_2: P \longrightarrow \mathbb{G}_m, \quad \begin{bmatrix}
        A_1 & * \\
        0 & A_2
    \end{bmatrix} \longmapsto \det(A_2).
    &&
\end{flalign*}
One checks that $\mathcal{O}(1)=F(E) \times_P k_{\det_2}$, $\det_2=\lambda_{n-r+1} + \cdots + \lambda_n$ in $X(T)$ and  $I=(i_1,\dots,i_r) \in W/W_P$ acts on $\det_2$ by $I\det_2 = \lambda_{i_1} + \cdots + \lambda_{i_r}$.

Let $h_{\mathcal{O}(1)}$ be the height function on $\operatorname{Gr}_r(E_K)$ induced by the model $(\operatorname{Gr}_r(E),\mathcal{O}(1))$. Let $I_0=(1,2,\dots,r)\in  W/ W_P$ be the longest Weyl element. By Theorem \ref{intro_ thm of ht filtration}, the essential minimum of $h_{\mathcal{O}(1)}$ is $\zeta_{\operatorname{ess}} = \langle \deg(F_Q),I_0\det_2 \rangle =\mu_1+\cdots+\mu_r$.

The only element in $\Delta \backslash \Delta_P$ is $\alpha = \lambda_{n-r}-\lambda_{n-r+1}$  and the function $\langle \alpha^\vee,\cdot \rangle$ on $N^1(\operatorname{Gr}_r(E))$ is given by $\mathcal{O}(1) \longmapsto \langle \alpha^\vee,\det_2 \rangle <0$ and $f \longmapsto 0$. $\langle \deg(F_Q), I_0\cdot \rangle$ on $N^1(\operatorname{Gr}_r(E))$ is given by $\mathcal{O}(1) \longmapsto \mu_1 + \cdots + \mu_r$ and $f \longmapsto 1$. By Theorem \ref{intro_ k-movable=ZhangMinima}, the big cone is given by $\langle \alpha^\vee,\cdot \rangle<0$ and $\langle \deg(F_Q), I_0\cdot \rangle>0$. One sees that it is the cone given by extremal rays $f$ and $\mathcal{O}(1)-(\mu_1+\dots+\mu_r)f$.

\subsection{Acknowledgements}
We thank Xinyi Yuan and Ye Tian for their consistent supports.
We thank Jie Liu, Marc Besson, Hongsheng Hu, Mihai Fulger and Brian Lehmann for discussions and email correspondences.

The idea of comparing height filtrations and base loci was proposed by Huayi Chen in a private conversation. The idea of computing concave transform was inspired from discussions with François Ballaÿ. The idea of proving Proposition \ref{weighted average equality} by Weyl symmetry was inspired from discussions with Tao Gui. We wish to thank them especially.

\section{Height filtrations and successive minima} \label{section_ ht fil and succ min}
Let $k$ be an algebraically closed field of characteristic zero and $C$ be a projective smooth curve over $k$ with function field $K=k(C)$. Let $G$ be a connected reductive group over $k$. Let $P \subseteq G$ be a parabolic subgroup and $\lambda: P \longrightarrow \mathbb{G}_m$ be a strictly antidominant character. Let $F$ be a principal $G$-bundle over $C$. Let $\mathcal{X}=F/P$ and $X=(F/P)_K$. Let $\mathcal{L}_\lambda = F \times_P k_\lambda$ and $h_{\mathcal{L}_\lambda}: X(\overline{K}) \longrightarrow \mathbb{R}$ be the induced height function.

\[ \xymatrix{
X=(F/P)_K \ar[d] \ar[r] & \mathcal{X}=F/P \ar[d] \\
\operatorname{Spek}(K) \ar[r] & C
} \]

Let $F_Q$ be the canonical reduction of $F$ to $Q$, and $\deg(F_Q) \in X(T)^\vee_\mathbb{Q}$ be the induced cocharacter. Let $W,W_P$ and $W_Q$ be the Weyl group of $G,L_P$ and $L_Q$. For $w \in W_Q \backslash W/W_P$, let $C_w=(F_Q \times_Q QwP/P)_K \subseteq X$ be the corresponding Schubert cell. The main result in this section is the following theorem:
\begin{thm} For any $t \in \mathbb{R}$, let $Z_t \subseteq X$ be the Zariski closure of the set $\big\{ x \in X(\overline{K}): h_{\mathcal{L}_\lambda}(x) < t \big\}$. Then
	\begin{flalign*}
    \quad \quad \quad &
    Z_t = X \Bigg\backslash \coprod_{\langle \deg (F_{Q }), w\lambda \rangle\geq t} C_w = \coprod_{\langle \deg (F_{Q }), w\lambda \rangle<t} C_w.
    &&
  \end{flalign*} \label{thm of ht filtration}
\end{thm}

\begin{cor}
	The filtration $\big\{ Z_t:t \in \mathbb{R} \big\}$ is finite, and is given by successively deleting Schubert cells $C_w$ with jumping points $\zeta_w=\langle \deg (F_{Q }), w\lambda \rangle$ for $w \in W_Q \backslash W/W_P$. \label{zeta_w successive minima}
\end{cor}

Let's call $\big\{ Z_t: t \in \mathbb{R} \big\}$ the \emph{height filtration}, and call its jumping points $\zeta_w$ the \emph{successive minima}. Note that our definition of successive minima are slightly different with Zhang \cite{Zhang_smallPoints}. Let $d=\dim G/P$, Zhang considers only dimension jumps $e_i=\inf\big\{ t: \dim Z_t \geq d-i+1 \big\}$, $i=1,\dots,d+1$. In the following, $e_i$ will be called \emph{Zhang's successive minima} to avoid ambiguity.

For $w \in W_Q \backslash W/W_P$, let
\begin{flalign*}
	\quad \quad \quad &
	\ell(w)=\max_{\sigma\in W_Q w W_P}\min_{\tau\in \sigma W_P}\ell(\tau)
	&&
\end{flalign*} be the length function, which is nothing but the dimension of the cell $C_w$.

\begin{cor}
	Zhang's successive minima are $e_i = \min \big\{\zeta_w: \ell(w) \geq d-i+1 \big\}$.
\end{cor}

\subsection{Vector bundles}
Let $E$ be a vector bundle on $C$. The \textit{degree} of $E$ is $\deg(E) := \deg(\det(E))$ and the \textit{slope} of $E$ is $\mu(E) := \deg(E)/\operatorname{rk}(E)$. It is called \textit{slope semistable} if for every subbundle $F$ of $E$, $\mu(F) \leq \mu(E)$. This is equivalent to $\mu(Q) \geq \mu(E)$ for every quotient bundle $Q$ of $E$.

There exists uniquely a filtration $0=E_0 \subseteq E_1 \subseteq \cdots \subseteq E_r=E$ such that
	\begin{enumerate}
		\item $E_i/E_{i-1}$ is semistable;
		\item $\mu(E_i/E_{i-1}) > \mu(E_{i+1}/E_{i})$.
	\end{enumerate}
This filtration is called the \emph{Harder-Narasimhan filtration} of $E$. $\mu(E_1)$ is called \emph{maximal slope} of $E$ and is also denoted by $\mu_{\max}(E)$. $\mu(E/E_{r-1})$ is called \emph{minimal slope} of $E$ and is also denoted by $\mu_{\min}(E)$.

\subsection{Principal bundles}
For any linear algebraic group $\Gamma/k$, let $X(\Gamma)=\mathrm{Hom}(\Gamma,\mathbb{G}_m)$ denote the character group of $\Gamma$. For a cocharacter $f \in X(\Gamma)^\vee$ and a character $\lambda \in X(\Gamma)$, we shall denote the pairing by $\langle f, \lambda \rangle$.
For any $\lambda \in X(\Gamma)$, denote by $k_{\lambda}$ the one-dimensional representation on the vector space $k$ with $\Gamma$ acting by $\lambda$.

A \emph{principal $\Gamma$-bundle} on $C$ is a variety $F$ equipped with a right action of $\Gamma$ and a $\Gamma$-equivariant smooth morphism $F \longrightarrow C$ such that the map
 \begin{flalign*}
\quad\quad\quad &F \times_C (C \times \Gamma) \longrightarrow F \times_C F,\quad (f,(x,g)) \longmapsto (f,fg) &
 \end{flalign*} is an isomorphism. 

Attached to a principal $\Gamma$-bundle $F$, one has an associated cocharacter
\begin{flalign*}
 \quad\quad\quad & \deg(F): X(\Gamma) \longrightarrow \mathbb{Z},\quad \lambda \longmapsto \langle \deg(F),\lambda \rangle=\deg(F \times_\Gamma k_\lambda).&
\end{flalign*} Here 
$\deg(F \times_\Gamma k_\lambda)$ is the degree of the line bundle $F\times_\Gamma k_\lambda$ on the curve $C$.

Let $H$ be a closed subgroup of a linear algebraic group $\Gamma$ over $k$. A \emph{reduction of structure group} of $F$ to $H$ is a pair $(F_H,\phi)$ where $F_H$ is a principal $H$-bundle and $\phi: F_H \times_H \Gamma \simeq F$ is an isomorphism.

By the universal property of the quotient $F/H$, the assignment to a section $\sigma: C \longrightarrow F/H$ the reduction $\sigma^*F$ of $F$ to $H$ is a one-one correspondence between reductions of structure group of $F$ to $H$ and sections of $F/H \longrightarrow C$.

\subsection{Reductive groups, characters and cocharacters}
Let $G$ be a connected reductive group over $k$. Fix a Borel subgroup $B\subseteq G$ and a maximal torus $T\subseteq B$. Let $W$ be the Weyl group and $\Delta$ be the set of simple roots with respect to $(G,B,T)$. For any $\alpha\in\Delta$, we denote by $\alpha^\vee$ the corresponding simple coroot.

We shall consider only parabolic subgroups containing $B$. For such a parabolic subgroup $P$, let $W_P\subseteq W$ be the Weyl group $W(L_P)$ of the Levi factor $L_P\subseteq P$ and $\Delta_P\subseteq \Delta$ be the simple roots of $L_P$. Note that the natural inclusion 
\begin{flalign*}
	\quad \quad \quad &
	X(P) \longrightarrow X(L_P) \longrightarrow X(Z(L_P))
	&&
\end{flalign*} becomes an isomorphism after tensoring with $\mathbb{Q}$. Thus we have
\begin{flalign*}
\quad\quad\quad & X(T)_\mathbb{Q} \longrightarrow X(Z(L_P))_\mathbb{Q}=X(P)_\mathbb{Q} &
\end{flalign*}
and by taking duals, we get the so-called \emph{slope map} $X(P)^\vee_\mathbb{Q} \longrightarrow X(T)^\vee_\mathbb{Q}$ introduced in \cite[\textsection 2.1.3]{schieder_hardernarasimhan_2015}. In other words, a cocharacter on $X(P)_\mathbb{Q}$ can be extended canonically to $X(T)_\mathbb{Q}$.

\subsection{The canonical reduction to a parabolic subgroup}
Let $G$ be a connected reductive group over $k$ and $F$ be a principal $G$-bundle over $C$. Let $F_P$ be a reduction of $F$ to a parabolic subgoup $P$. Let $\deg(F_P) \in X(T)_\mathbb{Q}^\vee$ be the induced (rational) cocharacter.

The $G$-bundle $F$ is called \emph{semistable} if for any parabolic subgroup $P$, any reduction $F_P$ of $F$ to $P$ and any dominant character $\lambda$ of $P$ which is trivial on $Z(G)$, we have $\langle \deg(F_P),\lambda \rangle \leq 0$.

Among all filtrations of a vector bundle, there is a canonical one (the Harder-Narasimhan filtration). Similarly, among all reduction to parabolic subgroups, there is a canonical one. A reduction $F_P$ of $F$ to a parabolic subgroup $P$ is called \emph{canonical} if the following two conditions hold:
	\begin{enumerate}
		\item The the principal $L_P$ bundle $F_P \times_P L_P$ is semistable.
		\item For any non-trivial character $\lambda$ of $P$ which is non-negative linear combination of simple roots, $\langle \deg(F_P), \lambda \rangle>0$.
	\end{enumerate}

Behrend in \cite{behrend_semi-stability_1995} proved canonical reduction exists uniquely. Note that when $G=\operatorname{GL}_n$, the above definition of semistability recovers Mumford's slope stability of vector bundles and the above deinition of canonical reduction recovers the Harder-Narasimhan filtration of vector bundles.

\subsection{Arakelov geometry on flag varieties over function fields}
Let $G$ be a connected reductive group over $k$ and $F$ be a principal $G$-bundle over $C$. Let $P \subseteq G$ be a parabolic subgroup and $\lambda: P \longrightarrow \mathbb{G}_m$ be a strictly antidominant character. Let $\mathcal{X}=F/P$ and $X=(F/P)_K$.

\[ \xymatrix{
X=(F/P)_K,L \ar[d] \ar[r] & \mathcal{X}=F/P,\mathcal{L}_\lambda \ar[d] \\
\operatorname{Spek}(K) \ar[r] & C
} \]
Let $\mathcal{L}_\lambda = F \times_P k_\lambda$, $L=\mathcal{L}_{\lambda,K}$ and $h_{\mathcal{L}_\lambda}: X(\overline{K}) \longrightarrow \mathbb{R}, x \longmapsto \frac{\overline{\{x\}} \cdot \mathcal{L}_\lambda}{\deg(x)}$ be the induced height function. Here $\overline{\{ x \}}$ is the Zariski closure of $x$ in $\mathcal{X}$.

For any $t \in \mathbb{R}$, let $Z_t \subseteq X$ be the Zariski closure of the set $\big\{ x \in X(\overline{K}): h_{\mathcal{L}_\lambda}(x)<t \big\}$. Let $\zeta_{\operatorname{ess}}(X) := \inf \big\{ t: Z_t = X \big\}$ be the \emph{essential minimum} of $h_{\mathcal{L}_\lambda}$ on $X$.

Let $F_Q$ be the canonical reduction of $F$ to $Q$, and $\deg(F_Q) \in X(T)^\vee_\mathbb{Q}$ be the induced cocharacter. Let $W,W_P$ and $W_Q$ be the Weyl group of $G,L_P$ and $L_Q$. For $w \in W_{Q } \backslash W / W_{P }$, write $\mathcal{C}_w=F_{Q } \times_{Q } QwP/P$, $\mathcal{X}_w=F_{Q } \times_{Q } \overline{QwP}/P$, $C_w=\mathcal{C}_{w,K}$ and $X_w = \mathcal{X}_{w,K}$.

Note that for any $w^\prime \in W_Q$ and $\lambda \in X(T)$, we have $w^\prime\lambda-\lambda\in\mathbb{Z}[\Delta_Q]$ \and consequently $\langle \deg(F_Q),w^\prime\lambda \rangle = \langle \deg(F_Q),\lambda \rangle$. Note also that for any $\lambda \in X(P)$ and $w \in W_P$, $w\lambda=\lambda$. So the number $\langle \deg(F_Q),w\lambda \rangle$ is well-defined for any $\lambda \in X(P)$ and $w \in W_Q \backslash W /W_P$.

$\mathcal{X}_w$ is a closed subscheme of $\mathcal{X}$, so $\mathcal{L}_\lambda|_{\mathcal{X}_w}$ induces a height function $h_{\mathcal{L}_\lambda}: X_w(\overline{K}) \longrightarrow \mathbb{R}$, which is nothing but the restriction of $h_{\mathcal{L}_\lambda}: X(\overline{K}) \longrightarrow \mathbb{R}$ to $X_w(\overline{K})$. Let $\zeta_{\operatorname{ess}}(X_w)$ be the essential minimum of $h_{\mathcal{L}_\lambda}$ on $X_w$.

Theorem \ref{thm of ht filtration} follows from the following two propositions.
\begin{prop} \label{height lower bound on Schubert cells}
 $h_{\mathcal{L}_\lambda}(x) \geq \langle \deg F_{Q }, w\lambda \rangle$ for any $x \in C_w(\overline{K})$.
\end{prop}

\begin{prop} \label{ess min on Schubert varieties}
	$\zeta_{\operatorname{ess}}(X_w)=\langle \deg (F_{Q }), w\lambda \rangle$.
\end{prop}

\begin{customthm}{\ref{thm of ht filtration}}
	For any $t \in \mathbb{R}$, let $Z_t \subseteq X$ be the Zariski closure of the set $\big\{ x \in X(\overline{K}): h_{\mathcal{L}_\lambda}(x) < t \big\}$. Then
	\begin{flalign*}
    \quad \quad \quad &
    Z_t = X \Bigg\backslash \coprod_{\langle \deg (F_{Q }), w\lambda \rangle\geq t} C_w = \coprod_{\langle \deg (F_{Q }), w\lambda \rangle<t} C_w.
    &&
  \end{flalign*}
\end{customthm}

\begin{proof}
	We write $\zeta_w = \langle \deg (F_{Q }), w\lambda \rangle$ and $h=h_{\mathcal{L}_\lambda}$ to simplify notations. For a closed subvariety $Y$ of $X$, let $Z_t(Y)$ be the Zariski closure in $Y$ of $\big\{ y \in Y(\overline{K}): h(y)<t \big\}$. By Proposition \ref{height lower bound on Schubert cells}, $Z_{\zeta_w}(X_w) \subseteq X_w \backslash C_w$, and we claim that $Z_{\zeta_w}(X_w) = X_w \backslash C_w=\bigcup_{w^\prime < w} X_{w^\prime}$.

    In fact, suppose conversely that $Z_{\zeta_w}(X_w) \subsetneq X_w \backslash C_w$. Then there exists $X_{w^\prime} \subsetneq X_w$ such that $Z_{\zeta_w}(X_w) \cap X_{w^\prime} \subsetneq X_{w^\prime}$. This forces $Z_{\zeta_w}(X_{w^\prime}) \subseteq Z_{\zeta_w}(X_w) \cap X_{w^\prime} \subsetneq X_{w^\prime}$. But on $X_{w^\prime}$, Proposition \ref{ess min on Schubert varieties} says the essential minimum is $\zeta_{w^\prime}<\zeta_w$, so $Z_{\zeta_w}(X_{w^\prime})=X_{w^\prime}$.  Contradiction! Thus on $X_w$, the set $Z_t(X_w)=X_w$ when $t>\zeta_w$ and $Z_t(X_w)=X_w \backslash C_w$ when $t \leq \zeta_w$.

    Let $w_0 \in W_Q \backslash W/W_P$ be the longest element. On $X=X_{w_0}$, when $t>\zeta_{w_0}$, $Z_t=X_{w_0}$ and when $t \leq \zeta_{w_0}$, $Z_t=X_{w_0} \backslash C_{w_0}=\bigcup_{w^\prime} X_{w^\prime}$ with $w^\prime<w_0$ and $\ell(w^\prime)=\ell(w_0)-1$. Repeating this procedure on each smaller $X_{w^\prime}$, we can complete the proof by induction.
\end{proof}

\subsection{Proof of Proposition \ref{height lower bound on Schubert cells}}
For $x\in X(K)$, let $\sigma_x : C \longrightarrow \mathcal{X}$ be the section induced by $x: \operatorname{Spec}(K) \longrightarrow X$ via valuative criterion, and let $F_{P,x}=\sigma_x^* F$ be the corresponding reduction  to $P$.

 \begin{lem} \label{height=degree}
For	$x \in X(K)$,  $h_{\mathcal{L}_\lambda}(x)=\langle \deg(F_{P,x}),\lambda \rangle$.
\end{lem}

\begin{proof}
By definition, $h_{\mathcal{L}_\lambda}(x)$ is the degree of $\sigma_x^* (\mathcal{L}_\lambda)$ and $\langle \deg(F_{P,x}),\lambda \rangle$ is the degree of $F_{P,x} \times_{P } k_\lambda$. We have equality of line bundles \begin{flalign*}
	   \quad\quad\quad & \sigma_x^* (\mathcal{L}_\lambda) = \sigma_x^*(F \times_{P } k_\lambda) = (\sigma_x^*F) \times_{P } k_\lambda = F_{P,x} \times_{P } k_\lambda
    &
	\end{flalign*} 
 and the lemma follows by taking degree.
\end{proof}

\begin{definition}[relative position]
    A reduction $F_P$ of $F$ to $P$ is called \emph{in relative position $w\in W_Q\backslash W/ W_P$ with respect to} $F_{Q}$ if the image of the natural map 
    \begin{flalign*}
   \quad\quad\quad &   F_{Q} \times_C F_{P} \longrightarrow  G_C ,\quad (a,b)\longmapsto a^{-1}b&  
    \end{flalign*} lies in $QwP_C$. Here we implicitly apply the injections 
    \begin{flalign*}
        \quad\quad\quad &F_P\hookrightarrow F_P\times_P G\simeq F,\quad F_Q\hookrightarrow F_Q\times_Q G\simeq F.
        &
    \end{flalign*} 
\end{definition}
Note that by the universal property of the quotient stacks $[G//Q\times P]$ and $[QwP//Q\times P]$, this definition coincides with the one in  \cite[\textsection 4.1]{schieder_hardernarasimhan_2015}.
  
\begin{lem}\label{relative position}
	When $x \in C_w(K)$, $F_{P,x}$ is of relative position $w$ with respect to $F_{Q}$.
\end{lem}
\begin{proof} Note that the injection $F_{P,x} = \sigma_x^*F \longrightarrow F$ has image in $F_Q \times_{Q} QwP$ by the commutativity of the following diagram
 \begin{flalign*}
    \xymatrix{
    \sigma_x^* F \ar[r] \ar[d] & F= \coprod F_Q \times_Q QwP \ar[d] \\
    C \ar[r]^{\sigma_x \quad \quad \quad \quad \quad} & F/P=\coprod F_Q \times_Q QwP/P.
    }
 \end{flalign*}
Now the lemma follows since $F_Q \times_C F_{P,x} \longrightarrow C \times G$ factors as 
\begin{align*}
    &F_Q \times_C F_{P,x} \longrightarrow  F_Q  \times_C \Big( F_Q \times_Q QwP \Big) \longrightarrow\\
    &\Big(F_Q \times_C F_Q\Big) \times_Q  QwP \longrightarrow (C \times Q) \times_Q QwP = C \times QwP.
\end{align*}
\end{proof}

\begin{prop} \label{Schieder 4.6}
    If $F_P$ is in relative position $w$ with respect to $F_Q$, then $\langle \deg(F_P),\lambda \rangle \geq \langle \deg(F_Q),w\lambda \rangle$  for any anti-dominant character $\lambda$.
\end{prop}

\begin{proof}
    This is written in the proof of \cite[Theorem 4.1]{schieder_hardernarasimhan_2015}. See \cite[\textsection 4.5.3]{schieder_hardernarasimhan_2015}.
\end{proof}

\begin{proof}[Proof of \emph{Proposition \ref{height lower bound on Schubert cells}}]
  Let $L$ be any finite field extension of $K$. Assume that $x \in C_w(L)$. Consider the following commutative diagram
    \[
    \xymatrix{
    X_L \ar[rr] \ar[dd] \ar[rd] && F_{C_L}/P \ar[rd] \ar[dd] \\
    & X \ar[dd] \ar[rr] && F/P \ar[dd] \\
    \operatorname{Spec}(L) \ar@/^1pc/[uu]^{\widetilde{x}} \ar[ru]^x \ar[rd] \ar[rr] && C_L \ar@/^1pc/[uu]^<<<<<<<<{\sigma_{\widetilde{x}}} \ar[ur]^{\sigma_x} \ar[rd] \\
    & \operatorname{Spec}(K) \ar[rr] && C.
    } \]
    where $C_L$ is the normalization of $C$ in $L$ and 
    $ F_{C_L}=F\times_C C_L$.
  Let   $\mathcal{L}_{C_L,\lambda}$ be the pullback of $\mathcal{L}_\lambda$ via the map $F_{C_L}/P \longrightarrow F/P$, which is also the line bundle associated to the character $\lambda$ of $P$ via the principal bundle $F_{C_L}$ on $C_L$. Then by Lemma \ref{height=degree},
  \begin{flalign*}
    \quad\quad\quad &
    h_{\mathcal{L}_\lambda}(x) = \tfrac1{[L:K]} \deg(\sigma_x^* \mathcal{L}_\lambda) = \tfrac1{[L:K]} \deg(\sigma_{\widetilde{x}}^* \mathcal{L}_{C_L,\lambda}) = \tfrac1{[L:K]} \langle \deg(F_{C_L,P,\widetilde{x}}),\lambda \rangle
    &
    \end{flalign*} 
where $F_{C_L,P,\widetilde{x}}$ is the reduction of $F_{C_L}$ to $P$ corresponding to the rational point $\widetilde{x} \in X_{L}(L)$.
    
    Note that the canonical reduction of $F_{C_L}$ is $F_{Q,C_L} = F_Q \times_C C_L$ and $\deg(F_{Q,C_L}) = [L:K] \deg(F_Q)$ in $X(T)_\mathbb{Q}^\vee$. By Lemma \ref{relative position},  $F_{Q,C_L}$ is of relative position $w$ with respect to $F_{C_L,P,\widetilde{x}}$. We conclude by applying Proposition \ref{Schieder 4.6} to $F_{C_L}$.
\end{proof}

\subsection{Proof of Proposition \ref{ess min on Schubert varieties}}
In this subsection, we employ Ballaÿ's theorem \cite[Theorem 1.2]{Ballay21} to compute the essential minimum. Let $\pi: F/P \longrightarrow C$ be the structure morphism. We start by finding the maximal slope of $\pi_* \mathcal{L}_{\lambda}|_{\mathcal{X}_w}$.

\begin{lem}[\cite{jantzen_representations_2007}, Part I, \textsection 5.18]
\label{principal bundle pushforward} 
On $\mathcal{X}_w=F_Q \times_Q \overline{QwP}/P\subseteq \mathcal{X}$, we have $\pi_* \big( \mathcal{L}_{\lambda} |_{\mathcal{X}_w} \big) = F_Q \times_Q H^0(\overline{QwP}/P,M_{\lambda})$.
\end{lem}
% Schieder computed the Harder-Narasimhan filtration of $F_Q \times_Q H^0(\overline{QwP}/P,M_{\lambda})$ in \cite[\textsection 5.1]{schieder_hardernarasimhan_2015}. Let $V$ be a finite dimensional $Q$-representation over $k$ with weight decomposition $V=\bigoplus_{\mu \in X(T)} V[\mu]$. The $Q$-filtration $V_{\bullet,\deg(F_Q)}$ of $V$ is defined by
% \begin{flalign*}
%     \quad \quad \quad &
%     V_{t,\deg(F_Q)} = \bigoplus_{\langle \deg(F_Q),\mu \rangle \geq t} V[\mu].
%     &&
% \end{flalign*}

% This filtration will be crucial not only in this section, but also in the computation of Boucksom-Chen concave transforms and the computation of base loci. The importance comes from the following proposition:

    Schieder computed the Harder-Narasimhan filtration of $F_Q \times_Q H^0(\overline{QwP}/P,M_{\lambda})$ in \cite[\textsection 5.1]{schieder_hardernarasimhan_2015}. To state his result, we fix the following notations which will be used frequently.
    \begin{definition}
        Let $V$ be a finite dimensional $Q$-representation over $k$ with weight decomposition $V=\bigoplus_{\mu \in X(T)} V[\mu]$. The $Q$-filtration $V_{\bullet,\deg(F_Q)}$ of $V$ is defined by
\begin{flalign*}
    \quad \quad \quad &
    V_{t,\deg(F_Q)} = \bigoplus_{\langle \deg(F_Q),\mu \rangle \geq t} V[\mu].
    &&
\end{flalign*}
    \end{definition}

\begin{prop} \label{thm_maxslop_schubert}
    The Harder-Narasimhan filtration of $F_Q \times_Q H^0(\overline{QwP}/P,M_\lambda)$ is $F_Q \times_Q H^0(\overline{QwP}/P,M_\lambda)_{\bullet,\deg(F_Q)}$. Moreover, the maximal slope is $\langle \deg (F_{Q }), w\lambda \rangle$. 
\end{prop}

\begin{proof}
The first assertion is a slight generalization of \cite[Proposition 5.1]{schieder_hardernarasimhan_2015},  whose original  texts only consider  $G$-representations restricted to $Q$, but the proof actually works  for arbitrary  $Q$-representations. For the second assertion, we only need to show  the highest weights in $H^0(\overline{QwP}/P,M_\lambda)$
    belong to $w\lambda+\mathbb{Z}[\Delta_Q]$.
   
When $Q=B$, $H^0(\overline{Bw P}/P,M_\lambda)$ has a single highest weight $w\lambda$   by the theory of Demazure modules \cite[Chapter 14]{jantzen_representations_2007}. For general $Q$,  let $\pi: W/W_P \longrightarrow W_Q \backslash W/W_P$ be the canoniacl projection. Note that
    \begin{flalign*}
        \quad \quad \quad &
        \overline{QwP}/P=\bigcup_{w^{\prime} \in \pi^{-1}(w)}\overline{Bw^{\prime}  P}/P .
        &&
    \end{flalign*} Since the restriction map
 \begin{flalign*}
    \quad \quad \quad &
        H^0(\overline{Qw P}/P,M_\lambda) \longrightarrow  H^0(\overline{Bw^{\prime} P}/P,M_\lambda) 
        &&
    \end{flalign*}
    is surjective for each $w^{\prime} \in \pi^{-1}(w)$ and the diagonal map
    \begin{flalign*}
    \quad \quad \quad &
        H^0(\overline{Qw P}/P,M_\lambda) \longrightarrow \prod_{w^{\prime} \in \pi^{-1}(w)} H^0(\overline{Bw^{\prime} P}/P,M_\lambda) 
        &&
    \end{flalign*}
 is injective, the highest weights in $H^0(\overline{QwP}/P,M_\lambda)$
    belong to $w\lambda+\mathbb{Z}[\Delta_Q]$.
\end{proof}

\begin{proof}[Proof of \emph{Proposition \ref{ess min on Schubert varieties}}]
By Ballaÿ's theorem \cite[Theorem 1.2]{Ballay21},
\begin{flalign*}
    \quad \quad \quad 
\zeta_{\operatorname{ess}}(X_w) & =\displaystyle \lim_{n \rightarrow \infty} \frac{\mu_{\max} (\pi_* \mathcal{L}_{n\lambda}|_{\mathcal{X}_w})}{n}\\
& = \lim_{n \rightarrow \infty} \frac{n\langle \deg (F_{Q }), w\lambda \rangle}{n} = \langle \deg (F_{Q }), w\lambda \rangle.
    &&
\end{flalign*}
\end{proof}

\section{String polytopes and the Boucksom-Chen concave transform}
Keep the notations in \textsection\ref{section_ ht fil and succ min}. Let $w_0 \in W$ be the longest Weyl element and fix a reduced decomposition $\underline{w_0}= s_{i_1} s_{i_2} \cdots s_{i_N}$, $N=\dim G/B$. Then for the antidominant weight $\lambda: P \longrightarrow \mathbb{G}_m$, we have the string valuation $\nu$ and the string polytope $\Delta_\lambda \subseteq \mathbb{R}^N$, which is a polytope of dimension $d=\dim G/P$.
Let $\mathcal{P}_\lambda$ be the weight polytope and $p: \Delta_\lambda \longrightarrow \mathcal{P}_\lambda$ be the weight map.

The main theorem of this section is the following:

\begin{thm} \label{concave transform}
	The Boucksom-Chen concave transform of ${\mathcal{L}_\lambda}$ (with respect to a specific valuation) is $\Delta_\lambda \longrightarrow \mathbb{R}, x \longmapsto \langle \deg(F_Q),p(x) \rangle$.
\end{thm}

Let $\mu$ be the Lebesgue measure on the real span of $\Delta_\lambda$, normalized with respect to $\mathbb{Z}^N$. As a direct corollary, we have by \cite{BouChen11} that
\begin{flalign*}
	\quad \quad \quad &
	h_{\mathcal{L}_\lambda}(X) = \frac{ \mathcal{L}_\lambda^{\dim \mathcal{X}} } {\dim \mathcal{X} \cdot L^{\dim X}} = \frac1{\operatorname{Vol}(\Delta_\lambda)} \int_{\Delta_\lambda} \langle \deg(F_Q),p(x) \rangle \, \mathrm{d}\mu(x).
	&&
\end{flalign*}

The integration can be computed via Weyl symmetry:
\begin{thm} \label{Zhang type equality}
    Let $q: W/W_P \longrightarrow W_Q \backslash W/W_P$ be the canoniacl projection. For $w \in W_Q \backslash W/W_P$, let $m_w=\# q^{-1}(w)$. Then 
   \begin{flalign*}
       \quad \quad \quad &
        h_{\mathcal{L}_\lambda} (X) = \frac{1}{\#W/W_P} \sum_{w\in W_Q \backslash W/W_P} m_w \langle \deg(F_Q),w\lambda \rangle .
       &&
   \end{flalign*}
\end{thm}

Note that $\zeta_w=\langle \deg(F_Q),w\lambda \rangle$ are successive minima of the height function $h_{\mathcal{L}_\lambda}$ on $X$ by Theorem \ref{thm of ht filtration}. Thus this theorem builds an equality between the variety height and a weighted average of the successve mimina. One may recognize this as a refinement of Zhang's inequality of successive minima.

\subsection{Okounkov bodies and the Boucksom-Chen concave transform}
Let $\mathbf{k}$ be a field and $Y$ be a normal projective variety of dimension $d$ over $\mathbf{k}$. Let $L$ be a big line bundle on $Y$, $V_n=H^0(Y,nL)$ and $V_\bullet = \bigoplus_{n \geq 0} V_n$. Let $\nu: V_\bullet \backslash \{0\} \longrightarrow \mathbb{Z}^r$ be a \emph{valuation with one dimensional leaves}, i.e.
\begin{itemize}
	\item for $a \in \mathbf{k}^\times$ and $s \in V_\bullet \backslash \{0\}$, we have $\nu(as)=\nu(s)$.
	\item for $s_1,s_2 \in V_\bullet \backslash \{0\}$, we have $\nu(s_1s_2)=\nu(s_1)+\nu(s_2)$ and $\nu(s_1+s_2) \geq \min\big\{ \nu(s_1),\nu(s_2) \big\}$.
	\item the quotient space $\big\{ s \in H^0(Y,nL): \nu(s) \geq a \big\} \Big/ \big\{ s \in H^0(Y,nL): \nu(s) > a \big\}$ is at most one dimensional.
\end{itemize}

Let $W_\bullet \subseteq V_\bullet$ be a graded subalgebra. We define $\Delta_\nu(W_\bullet) \subseteq \mathbb{R}^r$ to be the closure of the set $\bigcup_{n \geq 1} \tfrac1n \nu(W_n \backslash \{0\})$.
In particular, we define the \emph{Okounkov body} $\Delta_{\nu}(Y,L):=\Delta_\nu(V_\bullet)$.

A fundamental theorem is $\operatorname{Vol}(L) = d! \operatorname{Vol}(\Delta_\nu(Y,L))$ \cite[Theorem 1.15]{Kaveh_2015}, where the left hand side is the volume of the line bundle $L$ and the right hand side is the volume of the polytope $\Delta_\nu(Y,L)$.

Let $\big\{ V_\bullet^t: t \in \mathbb{R} \big\}$ be a decreasing $\mathbb{R}$-filtration of $V_\bullet$ that is \emph{multiplicative}, i.e. $V_{n_1}^{t_1} \cdot V_{n_2}^{t_2} \subseteq V_{n_1+n_2}^{t_1+t_2}$.
The \emph{Boucksom-Chen concave transform} associated to $\left(V_\bullet,\nu,\big\{ V_\bullet^t \big\} \right)$ is the function $G: \Delta_\nu(V_\bullet) \longrightarrow \mathbb{R} \cup \{-\infty\}$, $x \longmapsto \sup\big\{ t \in \mathbb{R}: x \in \Delta_\nu(V_\bullet^t) \big\}$ \cite{BouChen11}. $G$ is concave, as indicated in the name. In particular, $G$ is continuous in the interior.

\subsection{String polytopes and weight polytopes}
Let $w_0 \in W$ be the longest Weyl element and fix a reduced decomposition $\underline{w_0}= s_{i_1} s_{i_2} \cdots s_{i_N}$, $N=\dim G/B$. In \cite{Littelmann_1998,Bernstein-Zelevinsky}, the authors constructed the \emph{string valuation} $\nu=\nu_{\underline{w_0}}: H^0(G/P,nM_\lambda) \backslash \{0\} \longrightarrow \mathbb{Z}^N$. The image $\nu(H^0(G/P,nM_\lambda)\backslash\{0\})=\big\{v_1,\dots,v_l\big\}$ is in bijection (the \emph{string parametrization}) with the so-called \textit{dual crystal basis} $\big\{ e_{v_1}^{(n\lambda)},\cdots,e_{v_l}^{(n\lambda)} \big\}$ of $H^0(G/P,nM_\lambda)$ (see \cite[\textsection 3]{Kaveh_2015} for a detailed exposition). Moreover, the dual crystal basis is compatible with the weight decomposition
\begin{flalign*}
    \quad \quad \quad &
H^0(G/P,nM_\lambda)=\bigoplus_{\mu\in X(T)} H^0(G/P,nM_\lambda)[\mu],
    &&
\end{flalign*}
that is, for each $i=1,\dots,l$, $e_{v_i}^{(n\lambda)}\in H^0(G/P,nM_\lambda)[\mu]$ for some $\mu$.
% We denote the map $\Delta_{n\lambda}\cap\Z^N\rightarrow X(T),v_i\mapsto\mu$ by $p_n.$

The \emph{string polytope} $\Delta_\lambda$ is the Okounkov body $\Delta_\nu(G/P,M_\lambda)$, i.e. it is the closure of the set $\bigcup_{n \geq 1} \tfrac1n \nu(H^0(G/P,nM_\lambda) \backslash \{0\})$. Note that $\Delta_{\lambda} \cap \tfrac1n \mathbb{Z}^N$ is in bijection with the dual crystal basis of $H^0(G/P,nM_\lambda)$.

The \emph{weight polytope} $\mathcal{P}_\lambda \subseteq X(T)_\mathbb{R}$ is the convex hull of the Weyl orbit $W\lambda$ in $X(T)_\mathbb{R}$. Note that $\mathcal{P}_\lambda \cap \tfrac{1}{n} X(T)$ is in bijection with weights apprearing in $H^0(G/P,nM_\lambda)$. By taking weights, one has  \emph{weight map} $p: \Delta_\lambda \longrightarrow \mathcal{P}_\lambda$ which is affine and sends $\Delta_{\lambda} \cap \tfrac1n\mathbb{Z}^N$ to $\mathcal{P}_\lambda \cap \tfrac1n X(T)$ .

\subsection{Proof of Theorem \ref{concave transform}}
Let $V_n=H^0(G/P,nM_\lambda)$ and $V_\bullet = \bigoplus_{n \geq 0} V_n$. For $t \in \mathbb{R}$, let $V_n^t = H^0(G/P,nM_\lambda)_{t,\deg(F_Q)}$ as in Proposition \ref{thm_maxslop_schubert}, and $V_\bullet^t = \bigoplus_{n \geq 0} V_n^{nt}$. Let $\nu: V_\bullet \backslash \{0\} \longrightarrow \mathbb{Z}^N$ be the string valuation.
Let $\varphi: \Delta_\nu(G/P,M_\lambda) = \Delta_\lambda \longrightarrow \mathbb{R}$ be the Boucksom-Chen concave transform associated to $\left(V_\bullet,\nu,\big\{ V_\bullet^t \big\} \right)$.

\begin{thm}
	$\varphi(x)=\langle \deg(F_Q),p(x) \rangle$.
\end{thm}

% \begin{proof}
% 	$G$ is continuous on the whole $\Delta_\lambda$ by \cite[Theorem 1.1]{kronya2012functions}. Thus it suffices to prove $G(x)=\langle \deg(F_Q),p(x) \rangle$ for $x \in \Delta_\lambda \cap \tfrac1n \mathbb{Z}^N$ for every $n>0$.
	
% 	Let $e_{nx} \in H^0(G/P,nM_\lambda)$ be the vector corresponding to $nx \in \Delta_{n\lambda} \cap \mathbb{Z}^N$ via the string parametrization. Note that
% 	\begin{flalign*}
% 		\quad \quad \quad &
% 		e_{nx} \in H^0(G/P,nM_\lambda)[np(x)] \subseteq H^0(G/P,nM_\lambda)_{\langle \deg(F_Q),np(x) \rangle,\deg(F_Q)} = V_n^{n \langle\deg(F_Q),p(x) \rangle}.
% 		&&
% 	\end{flalign*} Thus tautologically, $nx$ is inside the image of the string valuation of $ V_n^{n\langle\deg(F_Q),p(x) \rangle}$. It follows that $x \in \Delta_{\nu}(V_\bullet^{\langle \deg(F_Q),p(x) \rangle})$ and $G(x) \geq \langle \deg(F_Q),p(x) \rangle$.
	
% 	The other direction $G(x) \leq \langle \deg(F_Q),p(x) \rangle$ is equivalent to $x \notin \Delta_\nu \left( V_\bullet^{\langle \deg(F_Q),p(x) \rangle+\varepsilon} \right)$ for any $\varepsilon>0$, i.e. $x$ is not in the closure of the set $\bigcup_{n \geq 1} \tfrac1n \nu \left( V_n^{n\langle \deg(F_Q),p(x) \rangle+n\varepsilon} \backslash \{0\} \right)$. 
% 	This can be tested by the function $\langle \deg(F_Q),p(-) \rangle$ it self. In fact, for any $y \in \bigcup_{n \geq 1} \tfrac1n \nu \left( V_n^{n\langle \deg(F_Q),p(x) \rangle+n\varepsilon} \backslash \{0\} \right)$, $\langle \deg(F_Q),p(y) \rangle \geq \langle \deg(F_Q),p(x) \rangle+\varepsilon$.
% \end{proof}

\begin{proof}
$V_n^{nt} = \bigoplus_\mu V_n[n\mu]$ for $\mu \in \mathcal{P}_{\lambda} \cap \tfrac1n X(T)$ with $\langle \deg(F_Q),n\mu\rangle\geq nt$, and hence $V_n^{nt} = \bigoplus_x k \cdot e_{nx}^{(n\lambda)}$ for $x \in \Delta_{\lambda}\cap \tfrac1n \Z^N$ with $\langle \deg(F_Q),p(nx)\rangle\geq nt$.
% \begin{flalign*}
%     \quad \quad \quad &
%     V_n^{nt}=\bigoplus_{\mu\in \mathcal{P}_{n\lambda} \cap X(T),\langle \deg(F_Q),\mu\rangle\geq nt} V_n[\mu]=\bigoplus_{x\in \Delta_{n\lambda}\cap \Z^N,\langle \deg(F_Q),p_n(x)\rangle\geq nt}e_x^{(n\lambda)}k.
%     &&
% \end{flalign*}
Therefore by definition,
\begin{flalign*}
    \quad \quad \quad
    \Delta(V_\bullet^t) &= \text{closure of } \bigcup\nolimits_n \big\{x\in \Delta_{\lambda}\cap \tfrac1n \Z^N: \langle\deg(F_Q),p(x)\rangle\geq t\big\} \\
    &= \big\{x\in \Delta_\lambda : \langle \deg(F_Q),p(x)\rangle\geq t\big\}
    &&
\end{flalign*} which concludes the proof.
\end{proof}

Consider the graded linear series $\widetilde{V}_\bullet = \bigoplus_{n \geq 0} H^0(X,nL)$. Fix an isomorphism $F_{Q,K} \simeq Q_K$. Then
\begin{flalign}
	\quad \quad &
	\widetilde{V}_\bullet = \bigoplus_{n \geq 0} \Big( F_Q \times_Q H^0(G/P,nM_\lambda) \Big)_K \simeq \bigoplus_{n \geq 1} H^0(G/P,nM_\lambda)_K. \label{iso111}
	&&
\end{flalign} Consider the filtrations of $\widetilde{V}_n = \Big( F_Q \times_Q H^0(G/P,nM_\lambda) \Big)_K$ induced by the Harder-Narasimhan filtration of $\pi_* \mathcal{L}_{n\lambda}= F_Q \times_Q H^0(G/P,nM_\lambda)$ as in Proposition \ref{thm_maxslop_schubert}
\begin{flalign*}
	\quad \quad \quad &
	\widetilde{V}_\bullet^t = \bigoplus_{n \geq 0} \Big( F_Q \times_Q H^0(G/P,nM_\lambda)_{nt,\deg(F_Q)} \Big)_K \simeq \bigoplus_{n \geq 0} \Big( H^0(G/P,nM_\lambda)_{nt,\deg(F_Q)} \Big)_K.
	&&
\end{flalign*}

Via the string parametrization of the dual crystal basis, the string valuation $\nu$ on $H^0(G/P,nM_\lambda)$ extends canonically to a valuation $\widetilde{\nu}$ on $H^0(G/P,nM_\lambda)_K$. This induces a valuation on $\widetilde{V}_n \simeq H^0(G/P,nM_\lambda)_K$ via the isomorphism (\ref{iso111}). We shall abuse notation and also call it $\widetilde{\nu}$.

The Boucksom-Chen concave transform of ${\mathcal{L}_\lambda}$ with respect to $\widetilde{\nu}$ is the one associated to $\left( \widetilde{V}_\bullet,\widetilde{\nu}, \big\{ \widetilde{V}_\bullet^t \big\} \right)$, and obviously it coincides with the one associated to $\left( V_\bullet,\nu, \big\{ V_\bullet^t \big\} \right)$. Theorem \ref{concave transform} follows.

\subsection{Proof of Theorem \ref{Zhang type equality}}
Let $\mu$ be the Lebesgue measure on the real span of $\Delta_\lambda$, normalized with respect to $\mathbb{Z}^N$. By \cite{BouChen11}, we have
\begin{flalign*}
	\quad \quad \quad &
	h_{\mathcal{L}_\lambda}(X) = \frac{ \mathcal{L}_\lambda^{\dim \mathcal{X}} } {\dim \mathcal{X} \cdot L^{\dim X}} = \frac1{\operatorname{Vol}(\Delta_\lambda)} \int_{\Delta_\lambda} \langle \deg(F_Q),p(x) \rangle \, \mathrm{d}\mu(x).
	&&
\end{flalign*}

\begin{prop}\label{weighted average equality}  For any linear function $f: X(T)_\mathbb{R} \longrightarrow \mathbb{R}$, we have
    \begin{flalign*}\quad\quad\quad &\frac1{\operatorname{Vol}(\Delta_\lambda)} \int_{\Delta_\lambda} f \circ p \; \mathrm{d}\mu  = \frac{1}{\#W/W_P} \sum_{w \in W/W_P} f(w\lambda).&&
    \end{flalign*}
\end{prop}

\begin{proof}[Proof of \emph{Theorem \ref{Zhang type equality}}]Take $f=\deg(F_Q)$ in Proposition \ref{weighted average equality}.
    \end{proof}

The rest of this subsection is devoted to the proof of Proposition \ref{weighted average equality}. We start with two lemmata.
\begin{lem}\label{Weyl invariance}The pushforward measure $p_*\mu$ on $\mathcal{P}_\lambda$ is $W$-invariant.
\end{lem}
\begin{proof} Note that the measure $\mu$ on $\Delta_\lambda$ is the limit of the Dirac measures $\delta_{\Delta_\lambda\cap \frac{1}{n}\mathbb{Z}^N}$, i.e. 
$\mu=\lim_n \delta_{\Delta_\lambda\cap \frac{1}{n}\mathbb{Z}^N}$.
Consider the map $\Delta_{n\lambda} \longrightarrow \mathcal{P}_{n\lambda}$ and note that it sends integral points to its corresponding weights. Dividing by n, we see that
\begin{flalign*}
  \quad\quad\quad &  p_{*} \delta_{\Delta_\lambda\cap\frac{1}{n}\mathbb{Z}^N} = \operatorname{mult}_n \delta_{\mathcal{P}_\lambda\cap \frac{1}{n}X(T)}, &
\end{flalign*}
where $\operatorname{mult}_n$ is the multiplicity function
\begin{flalign*}
\quad\quad\quad &
\operatorname{mult}_n: \mathcal{P}_\lambda\cap \tfrac{1}{n}X(T)\longrightarrow \mathbb{N},\quad \mu\longmapsto \dim H^0(G/P, M_{n\lambda})[n\mu].&
\end{flalign*}
Since both $\operatorname{mult}_n$ and $\delta_{\mathcal{P}_\lambda\cap \tfrac{1}{n}X(T)}$ are $W$-invariant, we deduce that
\begin{flalign*}
    \quad \quad \quad &
    p_*\mu=\lim_n p_*\delta_{\Delta_\lambda\cap \frac{1}{n}\mathbb{Z}^N}
    &&
\end{flalign*} is also $W$-invariant.
\end{proof}
\begin{lem}\label{Weyl average constancy}The map $\mathcal{P}_\lambda\longrightarrow \mathcal{P}_\lambda$, $\mu \longmapsto \sum_{w\in W}w\mu$
% \begin{flalign*}
%     \quad\quad\quad &
% \mathcal{P}_\lambda\longrightarrow \mathcal{P}_\lambda,\quad \mu \longmapsto \sum_{w\in W}w\mu &
% \end{flalign*}
is constant.
\end{lem}
\begin{proof}We claim that for any $\mu_1,\mu_2\in X(T)$ such that  $\mu_1|_{Z(G)}=\mu_2|_{Z(G)}$, \begin{flalign*}
    \quad\quad\quad &\sum_{w\in W}w\mu_1=\sum_{w\in W}w\mu_2. &
\end{flalign*}
In fact, for any simply connected cover $G^{sc}\longrightarrow G$, the preimage $T^{sc}$ of $T$ is a maximal torus in $G^{sc}$ and $W=W(G,T)=W(G^{sc}, T^{sc})$. Via the natural map $X(T)\longrightarrow X(T^{sc})$, we are reduced to the case $G$ is simply connected. In this case, $T^{W}=Z(G)$ and the desired equality follows directly.

$H^0(G/P, M_{\lambda})$ is a simple $G$-module so all weights $\mu$ occuring in $H^0(G/P, M_{\lambda})$ share a same restriction to $Z(G)$. So by the above claim, the lemma is true for lattice points $\mathcal{P}_\lambda \cap X(T)$. Note that the map $\mu \longmapsto \sum_{w\in W}w\mu$ is linear, we are done.
\end{proof}

\begin{proof}[Proof of \emph{Proposition \ref{weighted average equality}}]
Let $\mathcal{P}_{\lambda,+}$ be the intersection of $\mathcal{P}_\lambda$ with the anti-dominant cone $\big\{ \langle \alpha^\vee,\lambda \rangle<0 : \alpha \in \Delta\backslash \Delta_P \big\}$. Note that the subgroup $W_P\subseteq W$ acts trivially on $\mathcal{P}_\lambda$. By Lemma \ref{Weyl invariance}, one has 
\begin{flalign*}\quad\quad\quad & \int_{\Delta_\lambda} f \circ p \; \mathrm{d}\mu=\int_{\mathcal{P}_\lambda} f \; \mathrm{d}p_*\mu = \int_{\mathcal{P}_{\lambda,+}} \Big( \sum_{w \in W/W_P} f(wx) \Big) \; \mathrm{d} p_*\mu. &&
\end{flalign*}
By the linearity of $f$ and Lemma \ref{Weyl average constancy}, it moreover equals to 
\begin{flalign*}\quad\quad\quad & \int_{\mathcal{P}_{\lambda,+}} f \Big( \sum_{w \in W/W_P} wx \Big) \; \mathrm{d} p_*\mu=\Big( \sum_{w \in W/W_P} f(w\lambda) \Big)  \int_{\mathcal{P}_{\lambda,+}}1\; \mathrm{d} p_*\mu. &
\end{flalign*}
On the other hand, \begin{flalign*}\quad\quad\quad & \operatorname{Vol}(\Delta_\lambda)=\int_{\mathcal{P}_\lambda} 1 \; \mathrm{d} p_*\mu=\#{W/W_P} \int_{\mathcal{P}_{\lambda,+}} 1 \; \mathrm{d} p_*\mu. &
\end{flalign*}
Taking ratios, we obtain the desired equality.
\end{proof}

We end this section by the case $\operatorname{Gr}(2,4)$.
\begin{example}
Let $G=\operatorname{GL}_4$, $E$ be a vector bundle on $C$ of rank $4$ and consider the $\operatorname{Gr}(2,4)$-bundle $\operatorname{Gr}_2(E)$. In this case the string polytopes (with respect to a certain reduced decomposition) is just the Gelfand-Zetlin polytope \cite[Remark 2.4]{Kaveh_note_2011}. Identify $X(T)_\mathbb{Q}$ with $\bigoplus_{i} \mathbb{Q}\lambda_i$ as before. The line bundle $\mathcal{O}(1)$ corresponds to the character $\lambda=(0,0,1,1)$.

The Gelfand-Zetlin polytope $\operatorname{GZ}_\lambda$ is defined by $0 \leq x_2 \leq x_1 \leq x_3 \leq 1$ and $x_2 \leq x_4 \leq x_3$ in $\mathbb{R}^4$, with vertices $(0,0,0,0)$, $(0,0,1,0)$, $(0,0,1,1)$, $(1,0,1,0)$, $(1,0,1,1)$ and $(1,1,1,1)$. 
The weight polytope $\mathcal{P}_\lambda$ is the convex hull of $(0,0,1,1)$, $(0,1,0,1)$, $(0,1,1,0)$, $(1,1,0,0)$, $(1,0,1,0)$ and $(1,0,0,1)$, the weight map $p: \operatorname{GZ}_\lambda \longrightarrow \mathcal{P}_\lambda$ is
\begin{flalign*}
    \quad \quad \quad &
(x_1,x_2,x_3,x_4) \longmapsto (x_1,-x_1+x_2+x_3, -x_2-x_3+x_4,-x_4)+(0,0,1,1)
    &&
\end{flalign*} and $\deg(F_Q)$ is the function on $X(T)_\mathbb{Q}$ given by $\lambda_i \longmapsto \mu_i$ as in \textsection \ref{example_grassmann}.

A direct computation shows that
\begin{flalign*}
    \quad \quad \quad &
    \operatorname{Vol}(\operatorname{GZ}_\lambda)=\frac1{12}, \quad \int_{\operatorname{GZ}_\lambda} \deg(F_Q)\circ p \; \mr d \mu = \frac1{24} (\mu_1+\mu_2+\mu_3+\mu_4).
    &&
\end{flalign*}
Thus $h(X)=\frac12 (\mu_1+\mu_2+\mu_3+\mu_4)$.

On the other hand, $W/W_P= S_4 \big/ (S_2 \times S_2)$, $\#W/W_P=6$ and
\begin{flalign*}
    \quad \quad \quad &
    \sum_{w \in W/W_P} \langle \deg(F_Q),w\lambda \rangle = \sum_{1 \leq i < j \leq 4} \mu_i+\mu_j.
    &&
\end{flalign*} Thus $\frac{1}{\#W/W_P} \sum_{w \in W/W_P} \langle \deg(F_Q),w\lambda \rangle = \frac12 (\mu_1+\mu_2+\mu_3+\mu_4) = h(X)$, which is predicted by Theorem \ref{Zhang type equality}.
\end{example}

%----------------------------------------
\section{Base loci and movable cones}
Keep the notations in \textsection \ref{section_ ht fil and succ min}. We fix a closed point $p_0 \in C$ and set $f=\pi^*(\O_C(p_0))$. For any $w \in W_Q \backslash W/W_P$, let $\mathcal{C}_w = F_Q \times_Q QwP/P$. Let $t\in\Q$. The main theorem of this section is a computation of base loci.

\begin{thm} \label{prop_aug_bas}
Let $\mathrm{B}_+(\shfL_\lambda-tf)$ be the augmented base locus of $\shfL_\lambda-tf$. Then
\begin{flalign*}
    \quad \quad \quad &
\mr B_+(\shfL_\lambda-tf)=\scrX \Bigg \backslash \coprod_{\langle \deg(F_Q),w\lambda \rangle > t} \mc C_w = \coprod_{\langle \deg(F_Q),w\lambda \rangle \leq t} \mc C_w.
    &&
\end{flalign*}
\end{thm}

Comparing Theorem \ref{thm of ht filtration} and Theorem \ref{prop_aug_bas}, we find that $\mr B_+(\shfL_\lambda-tf)$ agrees with the Zariski closure of $Z_t$ in $\mathcal{X}$ for every $t \not= \zeta_w$. At critical points, $\mr B_+(\shfL_\lambda-tf)$ is upper continuous while $Z_t$ is lower continuous.

As a corollary, we compute the movable cones on flag bundles. (Please see \textsection \ref{subsection_ movable cones} for the notations.)
\begin{thm} \label{mov cone of flag bundle}
    The $k$-th movable cone $\operatorname{Mov}^k(F/P)$ is the cone defined by 
    \begin{enumerate}
        \item $\langle \alpha^\vee, \cdot \rangle <0$ for any $\alpha \in \Delta \backslash \Delta_P$, and
        \item $\langle \deg(F_Q),w \cdot \rangle>0$ for any $w \in W/W_P$ with $\ell(w) \geq n-k+1$.
    \end{enumerate}
\end{thm}

In particular, let $\mathcal{L}$ be a line bundle on $F/P$ and $e_{k}(h_{\mathcal{L}})$ be the $k$-th minimum of Zhang of the height function $h_{\mathcal{L}}$, we have
\begin{cor}\label{k-movable=ZhangMinima}
    $\mathcal{L}$ is $k$-movable if and only if $e_{k}(h_{\mathcal{L}})>0$.
\end{cor}

\subsection{Augmented base loci and the restricted volume}
Let $\mathbf k$ be a field and $Y$ be a normal projective variety over $\mathbf k$.  Let $L$ be a line bundle on $Y$ and $V \subseteq H^0(Y,L)$ be a subspace. We shall denote by $\mr{Bs}(V)$ the base locus of $V$.

The \textit{stable base locus} of $L$ is defined as $\mr{B}(L)=\bigcap_{n\geq 1} \mr{Bs}(H^0(Y,nL))$.  The \textit{augmented base locus} of $L$ is defined as $\mr{B}_+(L)=\bigcap_{n\geq 1} \mr B(nL-A)$ for any ample line bundle $A$. This is independent of the choice of $A$.

Let $Z$ be a closed subvariety of $Y$ of dimension $d$. The \textit{restricted volume} is defined as
    \begin{flalign*}
        \quad \quad \quad &
        \operatorname{Vol}(Y|Z,L)=\limsup_{n\rightarrow \infty}\frac{\dim_{\mathbf k}\mr{Im}\Big( H^0(Y,nL)\rightarrow H^0(Z,nL|_Z) \Big)}{n^d/d!}.
        &&
    \end{flalign*}
    
The following characterization of augmented base loci in terms of restricted volumes is well known, which we include here for  readers' convenience.
\begin{thm}[\cite{ELMNP2009resvol}, Theorem C] \label{thm_aug_bas_res_vol}
	$\mr B_+(L)$ is the union of all closed subvarieties $Z$ with $\operatorname{Vol}(Y|Z,L)=0$.
\end{thm}

Let $N^1(Y)$ be the Néron-Severi group of $Y$. It is proved in \cite{ELMNP2006} that augmented base loci are well defined on $N^1(Y)_\R:=N^1(Y)\ot_\Z\R$.

Let $\mc Y \longrightarrow C$ be a projective flat family over the curve $C$, and $\shfL$ be a line bundle over $\mc Y$. For any scheme theoretical point $p\in C$, we denote by $\kappa(p)$ the residue field of $p$. Let $\mc Z$ be a closed subvariety of $\mc Y$, we denote by $\mc Z_p$ the fiber $\mc Z\times_C \mr{Spec}(\kappa(p))$ of $\mc Z$ over $p$. We denote by $\shfL_p$ the restriction $\shfL|_{\mc Y_p}$ of $\shfL$ to $\mc Y_p$.

\begin{lem} $\mathrm{B}(\shfL)_p=\bigcap\nolimits_{n \geq 1} \mr{Bs}\left( \mr{Im}\Big(H^0(\mathcal{Y},n\shfL)\ot_k \kappa(p)\rightarrow H^0(\mc Y_p,n\shfL_p) \Big) \right)$.
\end{lem}

\subsection{Proof of Theorem \ref{prop_aug_bas}}
Let $M_\lambda = G \times_P k_\lambda$, which is an ample line bundle on $G/P$. Recall that we have the decomposition
\begin{flalign*}
     \quad \quad \quad &
H^0(G/P,M_\lambda)=\bigoplus_{\mu\in X(T)}H^0(G/P,M_\lambda)[\mu]
     &&
 \end{flalign*}
 and the filtration
 \begin{flalign*}
     \quad \quad \quad &
    H^0(G/P,M_\lambda)_{t,\deg(F_Q)} = \bigoplus_{\langle \deg(F_Q),\mu \rangle \geq t} H^0(G/P,M_\lambda)[\mu].
     &&
 \end{flalign*}
 
 We start with a lemma describing the vanishing behaviour of sections in $H^0(G/P,M_\lambda)[\mu]$.
\begin{lem}\label{Qu}
  Take $w \in W/W_P$. Then 
    \begin{enumerate}
        \item $s\neq0 \in H^0(G/P,M_\lambda)[w\lambda]$ $\Longrightarrow$ $s(x) \neq 0$ for any $x \in BwP/P$.
        \item $s \in H^0(G/P,M_\lambda)[\mu]$ for $\mu \not\leq w\lambda$ $\Longrightarrow$ $s(x)=0$ for any $x \in \overline{BwP}/P$.
    \end{enumerate}
\end{lem}
\begin{proof}Consider the restriction map 
\begin{flalign*}
\quad\quad\quad & \mathrm{Res}:\ H^0(G/P,M_\lambda)\longrightarrow H^0(\overline{BwP}/P, M_\lambda),&&
\end{flalign*}
which is $B$-equivariant and surjective. 
Since $H^0(G/P,M_\lambda)[w\lambda]$ and $H^0(\overline{BwP}/P, M_\lambda)[w\lambda]$ are both one-dimensional, one has $\mathrm{Res}(s)\neq0$ for any  $s\neq0 \in H^0(G/P,M_\lambda)[w\lambda]$.

By the closedness of the vanishing locus of $\mathrm{Res}(s)$, there exists $x\in BwP/P$ such that $\mathrm{Res}(s)(x)\neq0$. Note that by the highest weight theory, $B$ acts on $\mathrm{Res}(s)$ via $w\lambda$ and $BwP/P$ is a single $B$-orbit under translation. Consequently, $\mathrm{Res}(s)(x)\neq0$ for any $x\in BwP/P$.
 
As for $(2)$, for any $\mu \not\leq w\lambda$,  $H^0(\overline{BwP}/P, M_\lambda)[\mu]=0$. Consequently, $\mathrm{Res}(s)=0$ for any  $s \in H^0(G/P,M_\lambda)[\mu]$.
\end{proof}

Let $p$ be a scheme theoretical point of $C$ and consider the filtration $H^0(\scrX_p,\shfL_{\lambda,p})_{\bullet,\deg(F_Q)}$ of $H^0(\scrX_p,\shfL_{\lambda,p})$ given by \begin{flalign*}
\quad\quad\quad & H^0(\scrX_p,\shfL_{\lambda,p})_{t,\deg(F_Q)}=F_{Q,\kappa(p)} \times_Q H^0(G/P,M_\lambda)_{t,\deg(F_Q)}.
&& \end{flalign*} We have
\begin{lem}\label{lemma_schubert_base_loci}
$\operatorname{Bs} \Big( H^0(\scrX_p,\shfL_{\lambda,p})_{t,\deg(F_Q)} \Big) = \scrX_p \backslash \coprod_{\langle \deg(F_Q),w\lambda \rangle \geq t} \mc C_{w,p}$.
\end{lem}

\begin{proof}
Let $F_{T,\kappa(p)}$ be a reduction of $F_{Q,\kappa(p)}$ to $T$. For any $\mu\in X(T)$, consider the $\kappa(p)$-subspace $H^0(\scrX_p,\shfL_{\lambda,p})[\mu]=F_{T,\kappa(p)} \times_T H^0(G/P,M_\lambda)[\mu]$ of $H^0(\scrX_p,\shfL_{\lambda,p})=F_{T,\kappa(p)} \times_T H^0(G/P,M_\lambda)$. For $w^\prime\in W/W_P$, set 
$C_{w^\prime} = F_{T,\kappa(p)} \times_T Bw^\prime P/P$ and $X_{w^\prime} = F_{T,\kappa(p)} \times_T \overline{B w^\prime P}/P$. Then $\mc C_{w,p} = \coprod_{\pi(w^\prime)=w} C_{w^\prime}$ and $\scrX_{w,p} = \bigcup_{\pi(w^\prime)=w} X_{w^\prime}$ where
$\pi: W/W_P \longrightarrow W_Q \backslash W/W_P$ is the natural projection.  

Since the principal $T$-bundle $F_{T,\kappa(p)}$ is trivial, we deduce that $s(x) \neq 0$ for  any $x \in C_{w^\prime}$, any $w^\prime \in \pi^{-1}(w)$, and any $s \in H^0(\scrX_p,\shfL_{\lambda,p})[w^\prime \lambda]$ from Lemma \ref{Qu}(1). Thus if $\langle \deg(F_Q),w\lambda \rangle \geq t$, $H^0(\scrX_p,\shfL_{\lambda,p})[w^\prime \lambda] \subseteq H^0(\scrX_p,\shfL_{\lambda,p})_{t,\deg(F_Q)}$ for any $w^\prime \in \pi^{-1}(w)$ and consequently, $x \not\in \operatorname{Bs} ( H^0(\scrX_p,\shfL_{\lambda,p})_{\geq t})$ for any $x \in C_{w^\prime}$. This proves $\operatorname{Bs} ( H^0(\scrX_p,\shfL_{\lambda,p})_{t,\deg(F_Q)} ) \subseteq \scrX_p \backslash \coprod_{\langle \deg(F_Q),w\lambda \rangle \geq t} \mc C_{w,p}$.

    On the other hand, we claim that if $\langle \deg(F_Q),w\lambda \rangle < t$, then for any $s \in H^0(\scrX_p,\shfL_{\lambda,p})_{t, \deg(F_Q)}$, the restriction of $s$ to $X_w$ is zero. In fact, we may assume $s \in H^0(\scrX_p,\shfL_{\lambda,p})[\mu]$ for some $\mu$ and   $\langle \deg(F_Q),\mu \rangle \geq t$ implies that $\mu \not\leq w^\prime \lambda$ for any $w^\prime \in \pi^{-1}(w)$ (because if $\mu \leq w^\prime \lambda$ for some $w^\prime$, then $\langle \deg(F_Q),\mu \rangle \leq \langle \deg(F_Q),w\lambda \rangle<t$). By Lemma \ref{Qu}(2), the restriction of $s$ to $X_{w^\prime}$ is zero for each $w^\prime$. We win.
\end{proof}

\begin{customthm}{\ref{prop_aug_bas}}
	Let $B_+(\shfL_\lambda-tf)$ be the augmented base locus of $\shfL_\lambda-tf$. Then
\begin{flalign*}
    \quad \quad \quad &
\mr B_+(\shfL_\lambda-tf)=\scrX \Bigg \backslash \coprod_{\langle \deg(F_Q),w\lambda \rangle > t} \mc C_w = \coprod_{\langle \deg(F_Q),w\lambda \rangle \leq t} \mc C_w.
    &&
\end{flalign*}
\end{customthm}

\begin{proof}
    We may assume that $t\in \N$ by taking a sufficiently large tensor power of $\shfL_\lambda-tf$.
    Let $t_0=\min\{\langle \deg(F_Q),\mu\rangle: H^0(G/P,M_\lambda)[\mu]\not=0\text{ and }\langle \deg(F_Q),\mu\rangle>t\}$.
    % \begin{flalign*}
    %     \quad \quad \quad &
    %     t_0=\min\{\langle \deg(F_Q),\mu\rangle: H^0(G/P,M_\lambda)[\mu]\not=0\text{ and }\langle \deg(F_Q),\mu\rangle>t\}.
    %     &&
    % \end{flalign*}
    % $$t_0=\min\{\langle \deg(F_Q),\mu\rangle:\mu\in X(T),H^0(G/P,M_\lambda)[\mu]\not=0\text{ and }\langle \deg(F_Q),\mu\rangle>t\}.$$
    Take $E=F_Q\times_Q H^0(G/P, M_\lambda)_{t_0,\deg(F_Q)}\ot \O_C(-tx)$, which is a subbundle of $\pi_*(\shfL_\lambda-tf).$ Moreover, $\mu_{\min}(E)=t_0-t>0$, i.e. $E$ is ample.

    For each scheme-theoretical point $p\in C$, we have $E(p):=E\ot_C \kappa(p)=F_{Q,\kappa(p)}\times H^0(G/P,M_\lambda)_{t_0,\deg(F_Q)}.$
    The ampleness of $E$ gives that
    \begin{flalign*}
        \quad \quad \quad &
        H^0(C,\mr{Sym}^m E)\ot_k \kappa(p)\longrightarrow \mr{Sym}^m (E(p))
        &&
    \end{flalign*}
    % $$H^0(C,\mr{Sym}^m E)\ot_k K\rightarrow \mr{Sym}^m (V)$$
    is surjective for every $m\gg0$. Since $\pi_*(\shfL)\ot_C \kappa(p)=H^0(\scrX_p,\shfL_{\lambda,p})$, we have the commutative diagram
    \begin{flalign*}
    \xymatrix@C-=0.5cm{ H^0(C,\mr{Sym}^m E)\ot_k \kappa(p)  \ar[d] \ar[r] & \mr{Sym}^m(E(p)) \ar[r]&\mr{Sym}^m H^0(\scrX_p,\shfL_{\lambda,p}) \ar[d] \\
    H^0(C,\mr{Sym}^m\pi_*(\shfL_\lambda-tf))\ot_k \kappa(p)\ar[r] &H^0(C,\pi_*(m\shfL_{\lambda}-tmf))\ot_k \kappa(p) \ar[r] & H^0(\scrX_p,m\shfL_{\lambda,p})}
    \end{flalign*}
    which implies that $V:=\mr{Im}(\mr{Sym}^m E(p)\rightarrow H^0(\scrX_p,m\shfL_{\lambda,p}))$ is contained in the image $\mr{Im}(H^0(C,\pi_*(m\shfL_{\lambda}-tmf))\ot_k \kappa(p) \rightarrow H^0(\scrX_p,m\shfL_{\lambda,p}))=\mr{Im}(H^0(\scrX,m\shfL_{\lambda}-tmf)\ot_k \kappa(p) \rightarrow H^0(\scrX,m\shfL_{\lambda,p})).$
    Hence
    \begin{flalign*}
        \quad \quad \quad &
        \mr B(\shfL_\lambda-tf)_p \subseteq \mr{Bs}(V)=\mr{Bs}(E(p))= \scrX_p \backslash \coprod\nolimits_{\langle \deg(F_Q),w\lambda \rangle > t} \mc C_{w,p}
        &&
    \end{flalign*}
    by applying Lemma \ref{lemma_schubert_base_loci} to $H^0(\scrX_p,\shfL_{\lambda,p})_{t_0,\deg(F_Q)}$. Consequently,
    \begin{flalign} \label{eq_base_loc}
        \quad &
        \mr B(\shfL_\lambda-tf) \subseteq \scrX \backslash \coprod\nolimits_{\langle \deg(F_Q),w\lambda \rangle > t} \mc C_{w}.
        &&
    \end{flalign}
    % \begin{equation}\label{eq_base_loc}
    %     \mr B(\shfL_\lambda-tf)\subseteq \mr{Bs}(V')=\mr{Bs}(V)=X \backslash \coprod_{\langle \deg(F_Q),w\lambda \rangle \geq t_0>t} C_w.
    % \end{equation}
    
    Now we take an ample line bundle $\mc A$ over $\scrX$ of form $\shfL_{\lambda'}-t'f$ where $\lambda'\in X(T)$ is a strictly anti-dominant character and $\langle \deg(F_Q),\mu\rangle-t'>0$ for any $\mu\in X(T)$ such that $H^0(G/P,M_{\lambda'})[\mu]\not=0.$ 
    
    Take sufficiently large $n$ such that $\langle \deg (F_Q),w\lambda'\rangle-t'< n(t_0-t)$ for any $w\in W_Q\backslash W/W_P$. Thus $\langle \deg(F_Q),w(n\lambda-\lambda')\rangle>nt-t'$ if and only if $\langle \deg(F_Q),w\lambda\rangle\geq t_0.$ Similarly as in \eqref{eq_base_loc}, we obtain that 
    \begin{flalign*}
        \quad \quad \quad &
\mr B_+(\shfL_\lambda-tf)\subset\mr B(n(\shfL_\lambda-tf)-\shfL_{\lambda'}-t'f)\subseteq \scrX \Bigg\backslash \coprod_{\langle \deg(F_Q),w\lambda \rangle > t} \mc C_w.
        &&
    \end{flalign*}
    % $$\mr B_+(\shfL_\lambda-tf)|_X\subset\mr B(n(\shfL_\lambda-tf)-\shfL_{\lambda'}-t'f)|_X\subseteq X \backslash \coprod_{\langle \deg(F_Q),w\lambda \rangle > t} C_w.$$
    On the other hand, for any $w\in W_Q\backslash W/W_P$ such that $\langle \deg(F_Q),w\lambda\rangle\leq t$, since \begin{flalign*}
        \limsup_{m\rightarrow+\infty}\frac{\mu_{\max}(\pi_*(m(\shfL_{\lambda}-tf)|_{\scrX_w})}{n}=\zeta_1(\shfL_{\lambda}-tf|_{\scrX_w})=\zeta_w-t\leq 0,
    \end{flalign*}
    we have $\operatorname{Vol}(\shfL_{\lambda}-tf|_{\scrX_w})=0$ by a comparison result \cite[Theorem 2.4]{chen_maj} deduced from the Riemann-Roch theorem.
    Therefore $\operatorname{Vol}(\scrX|\scrX_w,\shfL_\lambda-tf)=0$, which implies that $\scrX_w\subseteq B_+(\shfL_\lambda-tf)$ due to Theorem \ref{thm_aug_bas_res_vol}. In conclusion, $ \mr B_+(\shfL_\lambda-tf)=\scrX\backslash \coprod_{\langle \deg(F_Q),w\lambda \rangle > t} \mc C_w.$
\end{proof}

\subsection{Movable cones and proof of Theorem \ref{mov cone of flag bundle}}
\label{subsection_ movable cones}
Let $Y$ be a projective smooth variety over $\mathbf{k}$. Recall 
\begin{itemize}
    \item The \textit{big cone} $\mathrm{Big}(Y)\subseteq N^1(Y)_\R$ is the cone of numerical classes of big divisors. Its closure is called \emph{Pseudo-effective cone}, denoted as $\mathrm{Psef}(Y)$.
    \item The \textit{ample cone} $\mathrm{Amp}(Y)\subseteq N^1(Y)_\R$ is the cone of numerical classes of ample divisors. Its closure is called \textit{nef cone}, denoted as $\mathrm{Nef}(Y)$.
    \item For each $k=1,\cdots, \dim Y$, the $k$-th \textit{movable cone} is defined as $\operatorname{Mov}^k(Y)= \big\{D\in N^1(Y)_{\mathbb{R}}:\mathrm{codim}(\mr B_+(D))\geq k\big\}$ which is a open in $N^1(Y)_\mathbb{R}$.
\end{itemize}

We have obvious inclusions $\operatorname{Mov}^d(Y) \subseteq \cdots \subseteq \operatorname{Mov}^1(Y)$. Note that $\operatorname{Mov}^1(Y)$ is just the big cone and $\operatorname{Mov}^d(Y)$ is just the ample cone.

Now we go back to the setting of \textsection \ref{section_ ht fil and succ min}. Denote by \begin{flalign*}
\quad\quad\quad &c:X(P)\longrightarrow \mathrm{Pic}(G/P),\quad \lambda\longmapsto M_\lambda& 
\end{flalign*}the character map. 
We have an exact sequence \cite[Theorem 18.32]{milne2017alggrp}:
\begin{flalign*}
\quad\quad\quad &0\longrightarrow X(G)\longrightarrow X(P)\longrightarrow \mathrm{Pic}(G/P)\longrightarrow \mathrm{Pic}(G)\longrightarrow \mr{Pic}(P)\longrightarrow 0.&
\end{flalign*}
Since $G$ and $P$ are connected linear algebraic groups over a characteristic zero field, both $\mathrm{Pic}(G)$ and $\operatorname{Pic}(P)$ are finite. Notice that the numerical equivalence and the linear equivalence coincide on $G/P$ since it is smooth and Fano. Therefore the linear map $X(P)_\R\rightarrow N^1(G/P)_\R$ is surjective. It is well-known that $M_\lambda$ is ample $\Longleftrightarrow$ $M_\lambda$ is big $\Longleftrightarrow$ $\lambda$ is strictly anti-dominant, thus
\begin{flalign*}
    \quad \quad \quad &
    \mathrm{Amp}(G/P)=\mathrm{Big}(G/P)=\big\{c(\lambda): \langle\alpha^\vee,\lambda\rangle<0 \text{ for all $\alpha\in \Delta \backslash \Delta_P$}\big\}
    && 
\end{flalign*}
% $$\mathrm{Amp}(G/P)=\mathrm{Big}(G/P)=\big\{c(\lambda):\lambda\in X(P)_\R, \langle\alpha^\vee,\lambda\rangle<0 \text{ for all $\alpha\in \Delta \backslash \Delta_P$}\big\},$$
where the pairing $\langle\cdot,\cdot\rangle$ is extended from $X(P)$ to $X(P)_\R$ by linearity. This implies $\mathrm{Mov}^i(G/P)=\mathrm{Amp}(G/P)$ for any $i=1,\cdots, \dim (G/P)$, which is the image of the strictly anti-dominant cone in $X(P)$ under $c: X(P)_\mathbb{R} \longrightarrow N^1(G/P)_\mathbb{R}$.

Then we compute $\mathrm{Mov}^i(F/P)$.
\begin{lem}
    Let $f \in N^1(F/P)$ be the class of a vertical fiber. Then every element in $N^1(F/P)$ can be writen as $\mathcal{L}_\lambda-tf$.
\end{lem}

\begin{proof}
Apply \cite[Theorem 18.32]{milne2017alggrp} to the $G$-bundle $F \longrightarrow C$. We get an exact sequence
\begin{flalign*}
    \quad \quad \quad &
    X(P) \longrightarrow \operatorname{Pic}(F/P) \longrightarrow \operatorname{Pic}(F) \longrightarrow \operatorname{Pic}(P).
    &&
\end{flalign*}
Apply it to the $P$-bundle $F \longrightarrow F/P$. We get another exact sequence
\begin{flalign*}
    \quad \quad \quad &
    X(G) \longrightarrow \operatorname{Pic}(C) \longrightarrow \operatorname{Pic}(F) \longrightarrow \operatorname{Pic}(G).
    &&
\end{flalign*} Note that after tensoring $\mathbb{R}$, $\operatorname{Pic}(P)_\mathbb{R}=\operatorname{Pic}(G)_\mathbb{R}=0$.

    Consider the diagram
    \begin{flalign*}
    \quad \quad \quad &
        \xymatrix{ 
        X(P)_\mathbb{R} \ar[r]^f & \operatorname{Pic}(F/P)_\mathbb{R} \ar[r]^g & \operatorname{Pic}(F)_\mathbb{R} \\
        & \operatorname{Pic}(C)_\mathbb{R} \ar[u]^\phi \ar[ru]^\psi
        } &&
    \end{flalign*} where $g$, $\phi$ and $\psi$ are just pullbacks. Note that $g$ and $\psi$ are surjective. Let $x \in \operatorname{Pic}(F/P)$ such that $g(x)=\psi(a)$ for some $a \in \operatorname{Pic}(C)_\mathbb{R}$. Then $g(x-\phi(a))=0$ and the first line is exact at the middle, so $x-\phi(a)=f(b)$.
\end{proof}

Thanks to the above description, we may define two functions $\langle \alpha^\vee,\cdot \rangle$ and $\langle \deg(F_Q),w\cdot \rangle$ on $N^1(F/P)_\mathbb{R}$ as
\begin{enumerate}
    \item $\langle \alpha^\vee,\cdot \rangle$ sends $\mathcal{L}_\lambda-tf$ to $\langle \alpha^\vee,\lambda \rangle$.
    \item $\langle \deg(F_Q),w\cdot \rangle$ sends $\mathcal{L}_\lambda-tf$ to $\langle \deg(F_Q),w\lambda \rangle-t$.
\end{enumerate}
 To see they are well-defined, it suffices to deal with the case $\mathcal{L}_\lambda$ equals to a multiple of $f$.  In this case, we see that $M_\lambda$ is a trivial line bundle on $G/P$ by restricting $\mathcal{L}_\lambda$ to a fiber. This implies $\lambda \in X(G)$ by the exact sequence $0 \longrightarrow X(G) \longrightarrow X(P) \longrightarrow \operatorname{Pic}(G/P)$. Then $\mathcal{L}_\lambda= F\times_P k_\lambda = \pi^*(F\times_G k_\lambda)$ where $\pi$ is the structure map $F/P \longrightarrow C$. In particular, $\mathcal{L}_\lambda = \langle\deg(F),\lambda \rangle f$ in $N^1(F/P)_{\mathbb{R}}$.
\begin{enumerate}
    \item $\langle \alpha^\vee,\cdot \rangle$ sends $f$ to zero and sends $\lambda \in X(G)$ also to zero. This proves $\langle \alpha^\vee,\cdot \rangle$ is well-defined.
    \item $\langle \deg(F_Q),w\cdot \rangle$ sends $\langle\deg(F),\lambda \rangle f$ to $\langle\deg(F),\lambda \rangle$ since $\lambda \in X(G)$ is invariant under Weyl group action  and sends $\lambda \in X(G)$ also to $\langle \deg(F_Q),w\lambda \rangle = \langle \deg(F_Q),\lambda \rangle = \langle\deg(F),\lambda \rangle$ since $F \times_G k_\lambda = F_Q \times_Q G \times_G k_\lambda = F_Q \times_Q k_\lambda$. This proves $\langle \deg(F_Q),w\cdot \rangle$ is well-defined.
\end{enumerate} 

\begin{customthm}{\ref{mov cone of flag bundle}}
	The $k$-th movable cone $\operatorname{Mov}^k(F/P)$ is the cone defined by 
    \begin{enumerate}
        \item $\langle \alpha^\vee, \cdot \rangle <0$ for any $\alpha \in \Delta \backslash \Delta_P$, and
        \item $\langle \deg(F_Q),w \cdot \rangle>0$ for any $w \in W/W_P$ with $\ell(w) \geq n-k+1$.
    \end{enumerate}
\end{customthm}

\begin{proof}
    The sufficiency follows immediately after Theorem \ref{prop_aug_bas}. For the necessity, let $\shfL_\lambda-tf$ be a line bundle on $F/P$ such that $\mr{codim}(\mr B_+(\shfL_\lambda-tf))\geq k$. Then in particular $\shfL_\lambda-tf$ is big. This impies $M_\lambda$ is big on $G/P$, which is equivalent to $\langle \alpha^\vee, \lambda \rangle<0$ for all $\alpha \in \Delta \backslash \Delta_P$.
    Then $\lambda$ is strictly anti-dominant and Theorem \ref{prop_aug_bas} applies, saying
    $\mr B_+(\shfL_\lambda-tf)=\scrX \backslash \coprod_{\langle \deg(F_Q),w\lambda \rangle -t > 0} \mc C_w$. It has $\operatorname{codim} \geq k$ $\Longleftrightarrow$ $\langle \deg(F_Q),w\lambda \rangle -t >0$ for all $w$ with $\ell(w) \geq n-k+1$.
\end{proof}

%----------------------------------------

%----------------------------------------
\bibliography{mybibliography}
\bibliographystyle{plain}

%----------------------------------------
\end{document}